\documentclass[11pt,a4paper]{amsart}

\usepackage{times,epsfig}
\usepackage{amsfonts,amscd,amsmath,amssymb,amsthm}

\usepackage{tikz}
\usepackage{xy} \xyoption{all}
\usepackage{color,graphics}

\usepackage{bbm}
\numberwithin{equation}{section}

\newcommand{\uloopr}[1]{\ar@'{@+{[0,0]+(-4,5)}@+{[0,0]+(0,10)}@+{[0,0] +(4,5)}}^{#1}}
\newcommand{\uloopd}[1]{\ar@'{@+{[0,0]+(5,4)}@+{[0,0]+(10,0)}@+{[0,0]+ (5,-4)}}^{#1}}
\newcommand{\dloopr}[1]{\ar@'{@+{[0,0]+(-4,-5)}@+{[0,0]+(0,-10)}@+{[0, 0]+(4,-5)}}_{#1}}
\newcommand{\dloopd}[1]{\ar@'{@+{[0,0]+(-5,4)}@+{[0,0]+(-10,0)}@+{[0,0 ]+(-5,-4)}}_{#1}}

\newcommand{\luloop}[1]{\ar@'{@+{[0,0]+(-8,2)}@+{[0,0]+(-10,10)}@+{[0, 0]+(2,2)}}^{#1}}

\DeclareSymbolFont{SY}{U}{psy}{m}{n}
\DeclareMathSymbol{\emptyset}{\mathord}{SY}{'306}

\DeclareMathSymbol{\newtimes}{\mathbin}{SY}{'264}

\newcommand{\R}{\mathbb{R}}

\newcommand{\C}{\mathbb{C}}
\newcommand{\Z}{\mathbb{Z}}
\newcommand{\N}{\mathbb{N}}

\newcommand{\cA}{{\mathcal A}}
\newcommand{\cB}{{\mathcal B}}
\newcommand{\cC}{{\mathcal C}}

\newcommand{\cF}{{\mathcal F}}
\newcommand{\cH}{{\mathcal H}}

\newcommand{\cK}{{\mathcal K}}
\newcommand{\cL}{{\mathcal L}}

\newcommand{\cO}{{\mathcal O}}

\newcommand{\cT}{{\mathcal T}}

\newcommand{\K}{\mathbb{K}}

\newcommand{\ol}{\overline}

\newcommand{\B}{\mathcal{B}}

\renewcommand{\1}{\mathbbm 1}

{\bf}{\it}
{\bf}{\it}
{\bf}{\it}
{\bf}{\it}

\newtheorem{theorem}{Theorem}[section]{\bf}{\it}
\newtheorem{proposition}[theorem]{Proposition}{\bf}{\it}
\newtheorem{corollary}[theorem]{Corollary}{\bf}{\it}
\newtheorem{lemma}[theorem]{Lemma}{\bf}{\it}
{\bf}{\it}
\newtheorem{problem}[theorem]{Problem}{\bf}{\it}
{\bf}{\it}

\theoremstyle{definition}
\newtheorem{definition}[theorem]{Definition}
\newtheorem{remark}[theorem]{Remark}
\newtheorem{example}[theorem]{Example}

\theoremstyle{plain}
\newcounter{theoremintro}
\newtheorem{thmintro}[theoremintro]{Theorem}
\newtheorem{introcorollary}[theoremintro]{Corollary}{\bf}{\it}
\theoremstyle{definition}

\DeclareMathAlphabet{\Ma}{U}{msa}{m}{n}
\DeclareMathAlphabet{\Mb}{U}{msb}{m}{n}

\DeclareMathAlphabet{\Meuf}{U}{euf}{m}{n}
\DeclareSymbolFont{ASMa}{U}{msa}{m}{n}
\DeclareSymbolFont{ASMb}{U}{msb}{m}{n}
\DeclareMathSymbol{\hrist}{\mathord}{ASMa}{"16}
\DeclareMathSymbol{\varkappa}{\mathalpha}{ASMb}{"7B}
\DeclareMathSymbol{\CrPr}{\mathord}{ASMb}{"6F}

\def\got#1{\Meuf{#1}}
\def\ot #1.{{\got{#1}}}

\title{Amenability and uniform Roe algebras}

\author[Pere Ara]{Pere Ara$^{1}$}
\address{Department of Mathematics,
  Universitat Aut\`onoma de Barcelona, 08193 Bellaterra (Bar\-ce\-lona), Spain}
\email{para@mat.uab.cat}

\author[Kang Li]{Kang Li$^{2}$}
\address{Department of Mathematics, University of M\"{u}nster,
  Einsteinstr. 62, 48149 M\"{u}nster, Germany}
 \email{lik@uni-muenster.de} 
  
\author[Fernando Lled\'o]{Fernando Lled\'o$^{3}$}
\address{Department of Mathematics, University Carlos~III Madrid,
  Avda.~de la Universidad~30, 28911 Legan\'es (Madrid), Spain
  and Instituto de Ciencias Matem\'{a}ticas (CSIC - UAM - UC3M - UCM).}
\email{flledo@math.uc3m.es}

\author[Jianchao Wu]{Jianchao Wu$^{4}$}
\address{Department of Mathematics, Penn State University, 109 McAllister Building, University Park, PA 16802, USA}
\email{jianchao.wu@psu.edu}

\date{\today}

\thanks{{$^{1}$} Supported by the grants DGI MICIIN MTM2011-28992-C02-01 and MINECO MTM2014-53644-P.}
\thanks{{$^{2}$} Supported by ERC Advanced Grant no.~OAFPG 247321, the Danish National Research Foundation through the Centre for Symmetry and 
Deformation (DNRF92) and the Danish Council for Independent Research (DFF-5051-00037).}
\thanks{{$^{3}$} Supported by Spanish Ministry of Economy and Competitiveness through projects DGI MTM2014-54692-P and
the {\em Severo Ochoa} Programe for Centers of Excellence in R\&D (SEV-2015-0554).}
\thanks{{$^{4}$} Supported by NSF grant DMS 1564281; previously supported by ERC Advanced Grant ToDyRiC 267079 and SFB 878 \emph{Groups, Geometry and Actions}.}


\begin{document}

\begin{abstract}
Amenability for groups can be extended to metric spaces, algebras over commutative fields and $C^*$-algebras by adapting the notion 
of F\o lner nets. In the present article we investigate the close ties among these extensions and show that these three pictures unify in the context of the uniform Roe algebra $C_\mathrm{u}^*(X)$ 
over a metric space $(X,d)$ with bounded geometry. In particular, we show that the following conditions are equivalent:
(1) $(X,d)$ is amenable; 
(2) the translation algebra generating $C_\mathrm{u}^*(X)$ is algebraically amenable
(3) $C_\mathrm{u}^*(X)$ has a tracial state; 
(4) $C_\mathrm{u}^*(X)$ is not properly infinite; 
(5) $[1]_0\neq [0]_0$ in the $K_0$-group $K_0(C_\mathrm{u}^*(X))$; 
(6) $C_\mathrm{u}^*(X)$ does not contain the Leavitt algebra as a unital $*$-subalgebra; 
(7) $C_\mathrm{u}^*(X)$ is a F\o lner $C^*$-algebra in the sense that it admits a net of unital completely positive maps into matrices 
which is asymptotically multiplicative in the normalized trace norm.
We also show that every possible tracial state of the uniform Roe algebra $C_\mathrm{u}^*(X)$ is amenable.
\end{abstract}

\maketitle

\tableofcontents

\section{Introduction}\label{sec:intro}

The notion of amenable groups was introduced by von Neumann in \cite{Neumann29} in his study of the Banach-Tarski paradox. It is defined by the existence of a finitely additive mean which is invariant under translations in the group. These groups correspond, precisely,
to those that can not be paradoxically decomposed, a result that is known as Tarski's alternative (see, e.g., \cite{T-49,dS-86,Wagon16}). 
F\o lner gave in \cite{Foelner55} another characterization of amenability corresponding to an internal approximation of the group in terms of finite
subsets that have controlled growth. More precisely, we say that a discrete group $\Gamma$ satisfies the F\o lner condition if for every
$\epsilon>0$ and every finite subset $E$ of $\Gamma$, there exists a non-empty finite subset $F$ of $\Gamma$ such that
 \begin{align*}
 \frac{|\gamma F \Delta F|}{|F|}<\epsilon\quad \text{for all } \gamma\in E \; ,
 \end{align*}
where $\Delta$ denotes the symmetric difference. Amenability has been shown to be a rich concept that quickly spread to other areas
such as harmonic analysis, ergodic theory, etc.~(see \cite{KL17}). 

Of particular interest for us in this article are the notions of amenability for bounded geometry metric spaces, algebras and operator algebras.
Amenability for metric spaces is defined by an isoperimetric condition in complete analogy to F\o lner's condition mentioned above (see
\cite{Block-Weinberger-92} and Definition~\ref{def:metric-amenability}). 
In this context one can prove many results similar to the case of groups. 
In particular, it is shown that Tarski's alternative regarding amenability vs.~paradoxical decompositions is still true for metric spaces and amenability is invariant under coarse equivalences among bounded geometry metric spaces (see \cite{Silberstein-Grigorchuk-Harpe-99} and references cited therein).

In the late nineties, Elek (\cite{Elek97}) was able to characterize amenability of a discrete group $\Gamma$ in terms of its algebraic uniform Roe algebra\footnote{It is called Gromov's translation algebra in Elek's paper.}
$C^*_\mathrm{u,alg}(\Gamma)$. It is well-known that the
algebraic uniform Roe algebra $C^*_\mathrm{u,alg}(\Gamma)$ is canonically isomorphic to the algebraic crossed product of $\ell^\infty(\Gamma)$ (the commutative algebra of bounded complex functions on $\Gamma$) by the group
$\Gamma$ with respect to the translation action (see, e.g., Lemma 4.30 in \cite{Roe03}). Elek used the F\o lner condition
of amenable groups and the paradoxical decomposition of non-amenable groups to show that a discrete group $\Gamma$ is amenable if and only if its
algebraic uniform Roe algebra $C^*_\mathrm{u,alg}(\Gamma)$ admits a tracial state if and only if $[1]_0\neq [0]_0$ in the algebraic $K$-group
$K_0(C^*_\mathrm{u,alg}(\Gamma))$. Moreover, it is shown by R\o rdam and Sierakowski in \cite{RS12} that a discrete group $\Gamma$ is non-amenable if
and only if the reduced $C^*$-crossed product $\ell^\infty(\Gamma)\rtimes_r \Gamma$ is properly infinite if and only if
$M_k(\ell^\infty(\Gamma)\rtimes_r \Gamma)$ is \emph{properly infinite} for some $k\in \N$. Recall that a unital $C^*$-algebra $A$ is properly infinite if there exist two isometries $x$ and $y$ in $A$ such that $xx^*+yy^*\leq 1$. This notion may be considered as a strong negation of the existence of tracial states, and an analogous definition can be made for general algebras over a commutative field. 

Elek introduced in \cite{Elek03} the notion of \emph{algebraic amenability} for finitely generated unital algebras over (commutative) fields.
His definition resembles F\o lner's characterization, 
with finite subsets replaced by finite-dimensional linear subspaces and cardinalities replaced by dimensions. In our paper
\cite{ALLW-1}, we studied this notion in the setting of general algebras (see Definition~\ref{def:alg-amenable2} below). Following the footsteps of R\o rdam and
Sierakowski, we showed that a metric space $X$ is amenable if and only if its associated translation algebra $\C_\mathrm{u}(X)$ is algebraically amenable if and
only if $\C_\mathrm{u}(X)$ is properly infinite if and only if $\C_\mathrm{u}(X)$ does not contain the Leavitt algebra as a unital $*$-subalgebra. 
In Section~\ref{sec:amenability-metric-algebras} we provide the basic definitions and results on metric space and algebraic amenability 
needed later. We refer the reader to \cite{ALLW-1} for proofs and further motivations.

The notion of F\o lner condition has also been considered in the $C^*$-algebraic setting (see, e.g., \cite{Connes76,ConnesIn76,Bedos95,AL14}). 
The first-named and third-named authors defined in \cite{AL14} the class of F\o lner $C^*$-algebras\footnote{This class coincides with the class
of weakly hypertracial $C^*$-algebras, cf., \cite{Bedos95}.}
as those unital $C^*$-algebras $\cA$ that admit a net of unital completely positive (u.c.p.) maps
 $\varphi_i\colon  \cA\to M_{k(i)}(\C)$ which is asymptotically multiplicative in the normalized trace norm, i.e.,
 \begin{equation*}
   \lim_i\|\varphi_i(ab)-\varphi_i(a)\varphi_i(b)\|_{2,\mathrm{tr}}=0\; , \quad \text{for all } a,b\in \cA \;,
 \end{equation*}
 where $\|x\|_{2,\mathrm{tr}}\colon=\sqrt{\mathrm{tr}(x^*x)}$, $x\in M_{n}(\C)$ and 
$\mathrm{tr}(\cdot)$ denotes the unique tracial state on $M_{n}(\C)$. It is shown in \cite{Bedos95} that the reduced $C^*$-crossed product 
$\ell^\infty(\Gamma)\rtimes_r \Gamma$ is a F\o lner $C^*$-algebra if and only if the group $\Gamma$ is amenable if and only if
$\ell^\infty(\Gamma)\rtimes_r \Gamma$ admits an amenable trace $\tau$, i.e., there exists a state $\psi$ on $\cL(\ell^2(\Gamma))$ which extends
$\tau$ and satisfies
$$
 \psi(TA)=\psi(AT), \quad \text{for all } T\in \cL(\ell^2(\Gamma)) \text{ and } A\in \ell^\infty(\Gamma)\rtimes_r \Gamma \; .
$$

In this article, we will give a unifying point of view to the different approaches to amenability as mentioned above. 
In order to do this, we generalize in Section~\ref{sec:Foelner-star} the notion of F\o lner $C^*$-algebras 
given in \cite{AL14} (see also \cite{Bedos95}) to semi-pre-$C^*$-algebras as introduced by Ozawa in \cite{Ozawa13}. 
This class of $*$-algebras includes (pre-)$C^*$-algebras and allows us to characterize at the algebraic level 
different analytic aspects of group $C^*$-algebras and uniform Roe algebras (see, e.g., Examples~\ref{ex:GA} and~\ref{ex:semi-Roe}). 
In Theorem~\ref{theorem:characFolner2}, we give a variety of characterizations of F\o lner semi-pre-$C^*$-algebras, which generalizes the main result
in \cite{AL14}. As an application of this result, we prove the following relation between algebraic amenability and F\o lner $C^*$-algebras.
\begin{thmintro}[cf., Theorem~\ref{theorem:alg-amen-to-Folner}]\label{thmintro-1}
 Let $\cA$ be a unital pre-$C^*$-algebra which is algebraically amen\-able. Then its closure $\overline{\cA}$ is a F\o lner $C^*$-algebra.
\end{thmintro}
However, the converse implication is in general not true:
a F\o lner $C^*$-algebra may contain a dense subalgebra not satisfying algebraic amenability (see Remark~\ref{rem:counter-ex}).

In section~\ref{sec:alg-amenable2}, we present a confluence point of all the approaches to amenability mentioned above and prove a unifying theorem that generalizes results in \cite{Elek97}, \cite{RS12} and \cite{ALLW-1}:
\begin{thmintro}[cf., Theorem~\ref{theorem:main1}]\label{thmintro-2}
Let $(X,d)$ be an extended metric space with bounded geometry, $\cA$ be a unital pre-$C^*$-algebra that contains the translation algebra
$\C_\mathrm{u}(X)$ 
as a dense $*$-subalgebra. Let $k,n$ be positive integers with $n\geq 2$.
Then the following conditions are equivalent: 
\begin{enumerate}
 \item $(X,d)$ is amenable.
 \item $C^*_\mathrm{u,alg}(X)$ is algebraically amenable. 
 \item $\cA$ is a F\o lner pre-$C^*$-algebra. 
 \item $\cA$ has a tracial state.
 \item $M_k(\cA)$ is not properly infinite. 
 \item $\cA$ does not contain the Leavitt algebra $L(1,n)$ as a unital $*$-subalgebra. 
 \item $\cA$ does not contain the Leavitt algebra $L_\infty$ as a unital $*$-subalgebra.
 \item $[1]_0\neq [0]_0$ in the (algebraic) $K$-group $K_0(\cA)$.
\end{enumerate}
\end{thmintro}

It is known that in the context of group $C^*$-algebras, the amenability of a discrete group $\Gamma$ is equivalent to a number of properties for $C^*(\Gamma)$, e.g., nuclearity, quasidiagonality, and the F\o lner property (cf., Example~\ref{ex:Rosenberg}). Uniform Roe algebras provide examples 
that show that, in general, nuclearity and F\o lner property are independent notions in $C^*$-algebras (see Remark~\ref{rem:property-A}).

In this section we also analyze the trace space of uniform Roe $C^*$-algebras. We show that for an amenable metric spaces with bounded geometry, every
tracial state on a uniform Roe $C^*$-algebra is amenable. 

\begin{introcorollary}[see Corollary~\ref{amenable trace}]
Let $(X,d)$ be an extended metric space with bounded geometry and denote by $\cA$ the uniform Roe algebra $C_\mathrm{u}^*(X)$ or the 
maximal uniform Roe-algebra $C_\mathrm{u,max}^*(X)$. Then $\cA$ has a tracial state if and only if $(X,d)$ is amenable. In this case every
tracial state of $\cA$ is amenable.
\end{introcorollary}

We conclude this section stating the problem if, in this context, every tracial state is quasidiagonal.
Finally, there is a natural strengthening of the notion of a F\o lner net which 
we call a {\em proper} F\o lner net. In the context of metric spaces and algebras, it corresponds to the 
F\o lner net being exhaustive, where its subtle difference from general F\o lner nets plays an important r\^{o}le in the study of amenability in these contexts (see \cite[Sections~2--4]{ALLW-1}). 
We introduce the new class of {\it properly F\o lner} $C^*$-algebras in Definition~\ref{def:properly-Foelner}, and  
prove analogs of Theorems~\ref{thmintro-1} and~\ref{thmintro-2} in this more restrictive context 
(see Theorems~\ref{theorem:alg-amen-to-Folner}~(ii) and~\ref{teo:Roe-proper-amenability}).

\textbf{Conventions:} Given sets $X_1$ and $X_2$, we denote their cardinalities by $|X_1|$ and $|X_2|$, and their 
disjoint union by $X_1\sqcup X_2$. We put $\mathbbm{N}_0=\{0,1,2,\dots\}=\N\sqcup\{0\}$.
We also denote by $\mathrm{tr}(\cdot)$ the unique tracial state on a matrix algebra
$M_{n}(\C)$ and by $\mathrm{Tr}(\cdot)$ the canonical semifinite trace on 
$\mathcal{L}{(\mathcal H)}$, the set of bounded linear operators on 
a complex Hilbert space $\cH$. All the representations of unital $*$-algebras are assumed to be unital.

\section{Basic definitions and results}\label{sec:amenability-metric-algebras}

In this section we will summarize the main definitions and results needed later.
First we recall the notion of amenability on metric spaces and $\C$-algebras. 
We refer to \cite{ALLW-1} (and references cited therein) for motivation, analysis or proofs in a more general context. 
In the last subsection we also include the definition and some properties of semi-pre-$C^*$-algebras
as introduced by Ozawa in \cite{Ozawa13}.

\subsection{Amenable metric spaces}\label{subsec:amenability-metric}

We begin by addressing the notion of amenability for a class of discrete metric spaces that
will be needed in the study of F\o lner type conditions for Roe $C^*$-algebras in Section~\ref{sec:alg-amenable2}.
See Section~2 in \cite{ALLW-1} for a detailed analysis of metric amenability.
We will consider {\em extended} metric spaces $(X,d)$, i.e., spaces where the metric is allowed to take the 
value $\infty$,
\[ 
d\colon X\times X \to [0, \infty] \;.
\]
The property that two points have a finite distance defines an equivalence relation, which decomposes $X$ uniquely into a disjoint union of
equivalence classes (which we call coarse connected components), i.e.,
$X= \bigsqcup_{i \in I} X_i$, such that each $(X_i, d|_{X_i \times X_i})$ is an ordinary metric space, while 
$d(X_i, X_j) = \infty$ for any different $i,j \in I$. 
In addition, we will restrict our attention to extended metric spaces $(X,d)$ that 
are \emph{uniformly discrete}, i.e., there exists a constant $C>0$ such that for any two distinct points $x,y\in X$ 
one has $d(x,y)\geq C$. This is no real restriction in the context of coarse geometry since every metric space is coarsely equivalent to a uniformly discrete one.
(see, e.g., \cite[Proposotion~1.3.8]{Nowak-Yu-12}). Following Definition~1.2.6 in \cite{Nowak-Yu-12}
we define the following class of metric spaces.

\begin{definition}\label{def:bd-geo}
A uniformly discrete metric space $(X,d)$ has \emph{bounded geometry}\footnote{Note that a metric space satisfying the condition stated in Eq.~(\ref{eq:bdd-geo}) is also called
	uniformly locally finite. Moreover, the name ``bounded geometry'' is used with slight variations
	in the literature (see, e.g., \cite[Section~5.5]{bBrown08}).} if 
for any radius $R>0$
\begin{equation}\label{eq:bdd-geo}
  \sup_{x\in X}|B_R(x)|<\infty \;,
\end{equation}
where $B_R(x):=\{y\in X:d(x,y)\leq R\}$ denotes the closed ball centered at $x$ with 
radius $R$.
\end{definition}

If $(X,d)$ is an extended metric space with bounded geometry, 
then each coarse connected component $X_i$ is discrete and countable, since a finite metric space 
is discrete and $X_i=\bigcup_{n\in \N} B_n(x_0)$ for any $x_0\in X_i$.

We recall next the definition of amenability for coarse metric spaces based on the notion of F\o lner sets. 
Let $(X,d)$ be an extended metric space and $A\subset X$. 
For any $R>0$ consider the following natural notions of boundaries of $A$:
\begin{itemize}
 \item {\em $R$-boundary:} $\partial_R A:=\{x\in X: d(x,A)\leq R\ \text{and}\ d(x,X\setminus A)\leq R \}$;
 \item {\em outer $R$-boundary:} $\partial^+_R A := \{x\in X\setminus A:d(x,A)\leq R\}$;
 \item {\em inner $R$-boundary:} $\partial^-_R A := \{x\in A:d(x,X\setminus A)\leq R\}$. 
\end{itemize}
Next we introduce the notion of amenability of metric spaces as in
\cite[Section~3]{Block-Weinberger-92}. The idea behind this notion of amenability 
can be traced back to Ahlfors' analysis of exhaustions of
an open surface in part~II of \cite{Ahlfors}.

\begin{definition}\label{def:metric-amenability}
Let $(X,d)$ be an extended metric space with bounded geometry. 
\begin{itemize}
 \item[(i)] Let $R>0$ and $\varepsilon\geq 0$. A finite non-empty set $F\subset X$ is called an $(R,\varepsilon)$-\emph{F\o lner set}
  if it satisfies
  \begin{equation*}
  \frac{|\partial_R F|}{|F|}\leq \varepsilon\;.
 \end{equation*}
We denote by $\mathrm{\mbox{F\o l}}(R, \varepsilon)$ the collection of $(R,\varepsilon)$-F\o lner sets.
 \item[(ii)] The metric space $(X,d)$ is called \emph{amenable} if for every $R>0$ and $\varepsilon > 0$ there exists
  $F\in \mathrm{\mbox{F\o l}}(R, \varepsilon)$.
 \item[(iii)] The metric space $(X,d)$ is called \emph{properly amenable} if for every $R>0$, $\varepsilon>0$ 
  and finite subset $A\subset X$ there exists
  a $F\in \mathrm{\mbox{F\o l}}(R, \varepsilon)$ with $A\subset F$.
\end{itemize}
\end{definition}

For a finitely generated discrete group $\Gamma$ equipped with a word length metric $d$, both notions of amenability for $(\Gamma, d)$ are equivalent to F\o lner's condition for the group $\Gamma$ (see e.g.,
\cite[Proposition~3.1.7]{Nowak-Yu-12}). Moreover, as in the usual metric space situation, it is obvious that if $X$ is finite then it is properly amenable by taking $F=X$. It is known that amenability is invariant under coarse equivalence between metric spaces with bounded geometry (see, e.g., 
\cite[Proposition~3.C.29]{Cornulier-Harpe-14} or \cite[Corollary~2.2 and Theorem~3.1]{Block-Weinberger-92}). 
For a detailed analysis of the relation between amenability and proper amenability
we refer to \cite[Section~2]{ALLW-1}.

The following lemma shows that proper amenability can be characterized
in terms of the cardinalities of the F\o lner sets.

\begin{lemma} \cite[Lemma 2.6]{ALLW-1}
\label{lem:proper-cardinality}
 Let $(X,d)$ be an infinite extended metric space with bounded geometry. 
 Then $X$ is properly amenable if and only if for every $R>0$, $\varepsilon>0$ and  $N\in\mathbb{N}$
 there exists an $F\in\mathrm{\mbox{F\o l}}(R,\varepsilon) $ such that 
 $|F| \geq N$.
\end{lemma}

Amenability for metric spaces comes, as in the case of groups, with an important dichotomy in relation to
paradoxical decompositions. For its formulation and for the construction of the 
translation algebra in Section~\ref{sec:alg-amenable2} we need to introduce the notion of a partial translation.

\begin{definition}\label{def:part-transl}
Let $(X,d)$ be an extended metric space. A {\em partial translation on} $X$ is a triple $(A,B,t)$, where $A,B\subset X$ 
and $t\colon A\rightarrow B$ is a bijection such that the graph of $t$ 
is controlled, i.e., 
\[
\sup_{x\in A}d(x,t(x))<\infty \;.
\]
We denote the corresponding
domain and range of $t$ by $\mathrm{dom}(t):=A$ and $\mathrm{ran}(t):=B$. Given two partial translations $t$ and $t'$, we may form their composition $t \circ t'$ by restricting the domain to $(t')^{-1}( \mathrm{dom}(t) \cap \mathrm{ran}(t') )$ and the range to $t( \mathrm{dom}(t) \cap \mathrm{ran}(t') )$. The set of all partial translations of $X$ is denoted as 
$\mathrm{PT}(X)$.

A mean $\mu$ on $(X,d)$ is a normalized, finitely additive map on the set of all subsets of $X$, $\mu\colon \mathcal{P}(X)\to [0,1]$.
The mean $\mu$ is called {\em invariant under partial translations} if $\mu(A)=\mu(B)$ for all partial translations $(A,B,t)$.
\end{definition}

Note that a partial translation can only map subsets of each coarse connected component $X_i$ onto subsets of $X_i$.
In this sense, each partial translations naturally decomposes along its coarse connected components.

\begin{definition}\label{def:paradoxical-decomposition}
Let $(X,d)$ be an extended metric space. A {\em paradoxical decomposition}\footnote{Actually, this definition is equivalent to a 
weak version of paradoxical decomposition of $(X,d)$ (see Remark 2.10 and Remark 2.12 in \cite{ALLW-1} for more details).} of $X$ is a (disjoint) partition 
$X=X_+\sqcup X_-$ such that there exist two partial translations $t_\pm\colon X\rightarrow X_\pm$.
\end{definition}

The following result gives some standard characterizations of amenable metric spaces that will be 
used later (see, e.g., \cite[Theorems~25 and 32]{Silberstein-Grigorchuk-Harpe-99}; see also 
Theorem~2.16 in \cite{ALLW-1} for a proof in extended metric space situation).

\begin{theorem}\label{theorem:amenable-metric}
Let $(X,d)$ be an extended metric space with bounded geometry. Then the following conditions are equivalent:
\begin{enumerate}
 \item \label{item:X-amen} $(X,d)$ is amenable.
 \item \label{item:X-no-parox} $X$ admits no paradoxical decomposition.
 \item \label{item:inv-meas} There exists a mean $\mu$ on $X$ which is invariant under partial translations.
\end{enumerate}
\end{theorem}

\begin{remark}\label{rem:A}
There is another non-equivariant analogue of amenability in the realm of coarse geometry, namely
property A introduced by Yu in \cite{Yu00}. 
This property guarantees the existence of a coarse embedding of the metric space into a Hilbert space
(see Chapters~4 and 5 of \cite{Nowak-Yu-12} and \cite{AGS12} for a precise definition and additional results).

We remark that these properties do not cover each other.
Any free group $\mathbb{F}_n$ with $n\geq 2$ generators is non-amenable (as a group and 
as a metric space) and has Property~A (cf., \cite[Example~4.1.5]{Nowak-Yu-12}).

To give an example of an amenable metric space without property~A we need to introduce first
the notion of box spaces (cf., \cite{Roe03,Kh12}).
Let $\Gamma$ be a finitely generated, residually finite group and let 
$\Gamma_1\supseteq \Gamma_2\supseteq\cdots \supseteq \Gamma_n\supseteq \cdots$ 
be a nested sequence of finite index normal subgroups of $\Gamma$ such that $\bigcap_{n=1}^\infty \Gamma_n=\{e\}$. 
The box space associated to the sequence $\{\Gamma_n\}_{n\in \N}$, denoted by $\text{Box}_{\Gamma_n}(\Gamma)$, is the disjoint union
$\bigsqcup_{n=1}^\infty \Gamma/\Gamma_n$ endowed with a metric $d$, such that: 
\begin{itemize}
\item[(i)] On each quotient $\Gamma/\Gamma_n$ the 
metric is the word metric with respect to the image of the finite generating set of $\Gamma$.
\item[(ii)] 
The distances between two finite quotients satisfy $d(\Gamma/\Gamma_n,\Gamma/\Gamma_m)\rightarrow \infty$ as $n+m\rightarrow \infty$ 
and $n\neq m$.
\end{itemize}

Any two metrics on $\text{Box}_{\Gamma_n}(\Gamma)$ satisfying the above two conditions are coarsely equivalent. Since each quotient 
group has a fixed number of generators, it is not hard to see that the box space has bounded geometry. Moreover, the box space is 
an amenable metric space. Indeed, 
for any $R>0$, there exists by condition (ii) above a natural number $n_R\in \N$ such that the $R$-boundary of $\Gamma / \Gamma_n$ 
is empty for all $n\geq n_R$.
It is well-known (see \cite[Proposition~11.39]{Roe03}) that $\Gamma$ is amenable if and only if $\text{Box}_{\Gamma_n}(\Gamma)$ has 
property A for any nested sequence $\Gamma_n$ of finite 
index normal subgroups of $\Gamma$ with trivial intersection.  In particular, 
the space $X:=\text{Box}_{\Gamma_n}(\mathbb{F}_2)$ is an amenable metric space 
(with bounded geometry) which does {\em not} have property A. 
\end{remark}

\begin{corollary}\label{cor:characamenable-nprpamen}
Let $(X,d)$ be a extended metric space with bounded geometry and $|X|=\infty$. 
Then $X$ is amenable but not properly amenable if and only if $X= Y_1\sqcup Y_2$, where $Y_1$ is a finite non-empty subset of $X$,
$Y_2$ is non-amenable and $d(x,y)= \infty $ for $x\in Y_1$ and $y\in Y_2$. 
\end{corollary}

\subsection{Algebraic amenability}\label{subsec:alg-amenable3}

In this subsection we will shortly review the main definitions and results concerning algebraic
amenability. For simplicity and coherence with following sections we will restrict 
our analysis of algebraic amenablity to algebras over the complex numbers. We will 
say $\cA$ is an algebra, meaning it is a $\C$-algebra.
Any result stated in this subsection for $\cA$ will be also true for $\K$-algebras over commutative fields $\K$. We refer to \cite[Sections~3 and 4]{ALLW-1} for 
proofs in the more general context of $\K$-algebras and additional motivation.
Our definition will follow existing notions in the literature
(see Section~1.11 in \cite{Gromov99} and \cite{Elek03,Cec-Sam-08}).

\begin{definition}\label{def:alg-amenable2}
 Let $\cA$ be an algebra. 
 \begin{itemize}
  \item[(i)] Let $\mathcal{F}\subset\cA$ be a finite subset and $\varepsilon \geq 0$. Then a nonzero finite-dimensional
   linear subspace $W \subset \cA$ is called a (left) \emph{$(\mathcal{F}, \varepsilon)$-F\o lner subspace} 
   if it satisfies
  \begin{equation}\label{eq:alg-amen2}
   \frac{\dim(a W +W)}{\dim(W)}\leq 1+\varepsilon\;,  
    \quad \text{for all } a\in\mathcal{F}\;.
  \end{equation}
  The collection of $(\mathcal{F}, \varepsilon)$-F\o lner subspaces of $\cA$ is denoted by $\mathrm{\mbox{F\o l}}(\cA, \mathcal{F}, \varepsilon)$.
  \item[(ii)] $\cA$ is (left) \emph{algebraically amenable} if for 
  any $\varepsilon >0$ and any finite set $\mathcal{F}\subset\cA$, there exists a left $(\mathcal{F}, \varepsilon)$-F\o lner subspace.
  \item[(iii)] $\cA$ is \emph{properly} (left) algebraically amenable if for any $\varepsilon >0$ and any finite set $\mathcal{F}\subset\cA$, 
  there exists a left $(\mathcal{F}, \varepsilon)$-F\o lner subspace $W$ such that $\mathcal{F} \subset W$. 
 \end{itemize}
\end{definition}

For brevity we are going to drop the term ``left'' for the rest of this section.
Any algebra satisfying $\dim(\cA)<\infty$ is obviously properly algebraically amenable by taking $W=\cA$.
Notice that although the definition works for algebras of arbitrary dimensions, the property of algebraic 
amenability is in essence a property for countably dimensional algebras (cf., Proposition~3.4 in \cite{ALLW-1}).

\begin{remark}
 The notion given by Elek in Definition~1.1 of \cite{Elek03} in fact corresponds to \emph{proper} algebraic amenability
 in the context of countably dimensional algebras as has been analyzed in \cite[Section~3]{ALLW-1}.
 \end{remark}

As proved in Proposition~3.4 of \cite{ALLW-1} we remark that, in essence, (proper) algebraic amenability
is a property of countably dimensional algebras:

\begin{proposition}\label{pro:alg-amenability-countability}
An algebra $\cA$ is (properly) algebraically amenable if and only if any countable subset in $\cA$ 
is contained in a countably dimensional subalgebra that is (properly) algebraically amenable.
\end{proposition}

The following example exhibits the difference between algebraic amenability and proper algebraic amenability.
(See also Theorem~3.2 in \cite{LledoYakubovich13} for an operator theoretic counterpart).

\begin{example} (\cite[Section~3]{ALLW-1})
Let $\cA$ be an algebra with a nonzero left ideal $I$ of finite dimension. Then $\cA$ is always algebraically
amenable, since $I$ is an $(\cA, \varepsilon=0)$-F\o lner subspace. Therefore an easy way to construct an amenable algebra that is not properly amenable is to take a direct sum of a finite dimensional algebra and a non-algebraically-amenable algebra (e.g., the group algebra of a non-amenable group; see Example~\ref{ex:group-algebra}). 
In particular, if $\cA$ is a non-amenable unital algebra, then $\widetilde{\cA} \cong \cA \oplus \mathbb{C}$ is algebraically amenable but not properly algebraically amenable.
Moreover, this is the only way in which a unitization $\widetilde{\cA}$ can be algebraically amenable but not properly algebraically amenable.
\end{example}

\begin{example} {\rm (\cite[Corollary 4.5]{Bart})}\label{ex:group-algebra}
The group algebra $\mathbb{C}G$ is algebraically amenable if and only if it is properly algebraically amenable if and only if 
$G$ is amenable.
\end{example}

\subsection{Semi-pre-$C^*$-algebras and their representations} \label{subsec:semi-prec}

To make this article as self-contained as possible we recall 
some definitions and basic facts about semi-pre-$C^*$-algebras and their representations (see Ozawa's recent paper 
\cite{Ozawa13} for more details and additional applications). 
This class of $*$-algebras include pre-$C^*$-algebras and allow to distinguish at the algebraic level between 
different analytic aspects of group $C^*$-algebras (see, e.g., Example~\ref{ex:GA}). 
In the following we only consider unital $*$-algebras 
over $\C$ and all homomorphisms and representations are assumed to be unital.

Given a unital $*$-algebra $\cA$, we denote by $\cA_h$ the set of all the hermitian (or self-adjoint) elements of $\cA$, i.e. $\cA_h = \{ a \in \cA \colon a = a^*\}$. 
It is clear that every element $X\in \cA$ can be uniquely written as a sum $X=A+iB$ of two hermitian elements $A$ and $B$. 
Moreover, the set of hermitian elements is an $\R$-vector space. We call a subset $\cA_+\subseteq \cA_h$ a 
\emph{$*$-positive cone}\footnote{This notion has been called \emph{quadratic module} in \cite{Sc09} and \emph{m-admissible wedges} 
in \cite{Sc90}.} if it satisfies the following conditions:
\begin{itemize}
\item[(i)] $\R_{\geq 0}\1\subseteq \cA_+$ and $\lambda A+B\in \cA_+$ for every $A,B\in \cA_+$ and $\lambda\in \R_{\geq 0}\;$.
\item[(ii)] $X^*AX\in \cA_+$ for every $A\in \cA_+$ and $X\in \cA$.
\end{itemize}
The \emph{algebraic $*$-positive cone}
\begin{align*}
\left\{\sum_{i=1}^n X_i^*X_i : X_1,\ldots, X_n\in \cA, n\in \N\right\}
\end{align*}
is obviously the smallest $*$-positive cone of $\cA$ and $\cA_h$ is the largest $*$-positive cone of $\cA$.
Given a $*$-positive cone $\cA_+$ of $\cA$, we write $A\leq B$ if $B-A\in \cA_+$ for $A,B\in \cA_h$. 
Then following \cite{Ozawa13}, we define the $*$-subalgebra of bounded elements by
\begin{align*}
\cA^\text{bdd}:=\text{\{$X\in \cA:\exists R>0$ such that $X^*X\leq R\1$\}}.
\end{align*}

\begin{definition}\label{def:semi-pre}
A unital $*$-algebra $\cA$ is called a semi-pre-$C^*$-algebra if it is equipped with a distinguished $*$-positive cone $\cA_+$ 
satisfying the Combes axiom that $\cA=\cA^\text{bdd}$.
\end{definition}
For a semi-pre-$C^*$-algebra $\cA$, one has $\cA_h=\cA_+-\cA_+$. Indeed, $H=\frac{\1+H^2}{2}-(\frac{\1+H^2}{2}-H)$ 
for every $H\in \cA_h$. 
We define the \emph{ideal of infinitesimal elements} of a semi-pre-$C^*$-algebra $\cA$ by
\begin{align*}
I(\cA)=\{X\in \cA:X^*X\leq \varepsilon \1\ \text{for all $\varepsilon>0$}\}
\end{align*}
and the \emph{archimedean closure} of any $*$-positive cone $\cA_+$ of $\cA$ by
\begin{align*}
\text{arch}(\cA_+)=\{A\in \cA_h:A+\varepsilon \1\in \cA_+\ \text{for all $\varepsilon>0$}\}.
\end{align*}
It is clear from the definition and the Combes axiom that the archimedean closure $\text{arch}(\cA_+)$ is again a 
$*$-positive cone which contains $\cA_+$. Moreover, $\text{arch}(\cA_+) \cap (-\text{arch}(\cA_+))\subseteq I(\cA)$. The cone $\cA_+$ is said to be \emph{archimedean closed} if $\cA_+=\text{arch}(\cA_+)$.

In the following, we will only consider positive $*$-representations of semi-pre-$C^*$-algebras, i.e., $*$-representations $\pi\colon \cA\to \cL(\cH)$ such~that $\pi(\cA_+)\subset \cL(\cH)_+$. Contrary to the $C^*$-algebraic setting this is not automatic for semi-pre-$C^*$-algebras. A positive
$*$-representation $\pi\colon \cA\to \cL(\cH)$ is faithful if for any positive $A\in\cA_+$
such that $\pi(A)=0$, it follows that $A=0$.

An important link to $C^*$-algebra theory is the notion of the universal $C^*$-algebra.

\begin{definition}\label{def:universal-C*}
The universal $C^*$-algebra of a semi-pre-$C^*$-algebra $\cA$ is the unital $C^*$-algebra $C^*_\mathrm{univ}(\cA)$ together with a positive 
$*$-homomorphism $\iota\colon \cA\rightarrow C^*_\mathrm{univ}(\cA)$ which satisfies the following properties: $\iota(\cA)$ is dense 
in $C^*_\mathrm{univ}(\cA)$ and every positive $*$-representation $\pi$ of $\cA$ on a Hilbert space $\cH$ extends to a 
$*$-representation $\widehat{\pi}\colon C^*_\mathrm{univ}(\cA)\rightarrow \cL(\cH)$, i.e., $\pi=\widehat{\pi}\circ \iota$.
\end{definition}

It is not hard to see that such a universal $C^*$-algebra is indeed unique and $C^*_\mathrm{univ}(\cA)$ is the separation and completion of the semi-pre-$C^*$-algebra $\cA$ under the $C^*$-semi-norm
\begin{align*}
\sup\Big\{\|\pi(A)\|_{\cL(\cH)}:\text{$\pi$ is a positive $*$-representation on a Hilbert space $\cH$}\Big\}.
\end{align*}
Hence every positive $*$-homomorphism between semi-pre-$C^*$-algebras extends uniquely to a (positive) $*$-homomorphism 
between their universal $C^*$-algebras. It may happen that $\cA_+=\cA_h$ and $C^*_\mathrm{univ}(\cA)=\{0\}$, 
which is still considered as a unital $C^*$-algebra.
\begin{theorem}[{\cite[Proposition~15]{Sc09}} and {\cite[Theorem~1]{Ozawa13}}]\label{semi}
Let $\cA$ be a semi-pre-$C^*$-algebra and $\iota\colon\cA\rightarrow C^*_\mathrm{univ}(\cA)$ be the positive $*$-homomorphism into 
the universal $C^*$-algebra of $\cA$. Then the following results hold:
\begin{itemize}
\item[(1)] The ideal of the infinitesimal elements  $I(\cA)$ is equal to $\ker \iota\;$.
\item[(2)] The archimedean closure of the positive cone $\emph{arch}(\cA_+)$ is equal to $\cA_h\cap \iota^{-1}(C^*_\mathrm{univ}(\cA)_+)\;$.
\item[(3)] The $C^*$-norm  $\|\iota(X)\|_{C^*_\mathrm{univ}(\cA)}$ is given by $\inf \{R>0 : R^2\1-X^*X\in \cA_+\}\;$.
\end{itemize}
\end{theorem}

The universal $C^*$-algebra $C^*_\mathrm{univ}(\cA)$ can also be obtained as the closure of the image under the universal positive $*$-representation,  
which comes from the GNS construction. Recall that a linear functional on a semi-pre-$C^*$-algebra $\phi\colon\cA\rightarrow \C$ 
is called a \emph{state} if $\phi$ is positive, self-adjoint and unital. We denote that set of states on $\cA$ by $S(\cA)$.
For any $\phi\in S(\cA)$ we write the corresponding positive cyclic GNS $*$-representation as
\[
 \pi_\phi\colon\cA\rightarrow \cL(L^2(\cA,\phi))\;.
\] 
Then we have that the universal $C^*$-algebra $C^*_\mathrm{univ}(\cA)$ of $\cA$ is the closure of the image 
under the \emph{universal positive $*$-representation}
 \begin{align*}
 \pi_\mathrm{univ} := \bigoplus_{\phi\in S(\cA)} \pi_{\phi}\colon \cA \rightarrow \cL \left(\bigoplus_{\phi\in S(\cA)} L^2(\cA,\phi)\right)\;,
 \end{align*}
(See, e.g., \cite[Section~6]{Ozawa13} for details).

Finally, we introduce the notion of maximal representation of a semi-pre-$C^*$-algebra.
\begin{definition}\label{def:maximal}
A positive unital $*$-representation $\pi_m\colon \cA \rightarrow \cL(\cH_m)$ is \emph{maximal} if for any
unital positive $*$-representation $\pi\colon \cA \rightarrow \cL(\cH)$ one has
\[
 \|\pi(A)\|\leq \|\pi_m(A)\|\;,\quad \text{for all } A\in\cA\;.
\]
\end{definition}

Given two unital positive $*$-representations on a semi-pre-$C^*$-algebra $\pi\colon \cA \rightarrow \cL(\cH_\pi)$ and 
$\rho\colon \cA \rightarrow \cL(\cH_\rho)$, we say that $\pi$ weakly contains $\rho$ if $\|\rho(A)\|\leq \|\pi(A)\|$ 
for every $A\in \cA$. In particular, since every $*$-homomorphism on a $C^*$-algebra 
$\cB$ is automatically contractive, we see that $\pi\colon\cB\to \cL(\cH_\pi)$ 
weakly contains $\rho\colon\cB\to \cL(\cH_\rho)$ if and only if $\ker \pi\subseteq \ker \rho$ provided that $\cB$ is a $C^*$-algebra
(see (3.4.5) in \cite[Chapter~3, \S~4]{Dixmier77}). It is clear that the universal positive $*$-representation $\pi_\mathrm{univ}$ 
weakly contains all positive $*$-representations, hence it is maximal. Moreover, any positive $*$-representation 
weakly containing the universal positive $*$-representation is also maximal. 

Next we show the relation between maximal representations of semi-pre-$C^*$-algebras and faithful representations 
of the corresponding universal $C^*$-algebras.

\begin{lemma}\label{lem:max-faithful}
Let $\cA$ be a semi-pre-$C^*$-algebra and denote by $C^*_\mathrm{univ}(\cA)$ the corresponding universal $C^*$-algebra 
which we assume to be nonzero. Let $\pi_m\colon \cA \rightarrow \cL(\cH_m)$ be a positive $*$-representation and let $\widehat{\pi_m} \colon C^*_\mathrm{univ}(\cA)\to \cL(\cH_m)$ be the corresponding $*$-representation of the universal $C^*$-algebra in the sense that $\pi_m=\widehat{\pi_m}\circ\iota$. Then $\pi_m\colon \cA \rightarrow \cL(\cH_m)$ is a maximal representation if and only if $\widehat{\pi_m} \colon C^*_\mathrm{univ}(\cA)\to \cL(\cH_m)$ is faithful.
\end{lemma}
\begin{proof}
	By Definition~\ref{def:universal-C*}, $\iota(\cA)$ is dense in $C^*_\mathrm{univ}(\cA)$. It follows by approximation that $\pi_m$ is maximal among all positive $*$-representation of $\cA$ if and only if $\widehat{\pi_m}$ is maximal among all positive $*$-representation of $C^*_\mathrm{univ}(\cA)$, i.e., $\widehat{\pi_m}$ weakly contains any $*$-representation $\rho$ of $C^*_\mathrm{univ}(\cA)$. Since $C^*_\mathrm{univ}(\cA)$ is a $C^*$-algebra, the latter condition is equivalent to that $\operatorname{ker}(\widehat{\pi_m}) \subset \operatorname{ker}(\rho)$ for any $*$-representation $\rho$ of $C^*_\mathrm{univ}(\cA)$ (see (3.4.5) in \cite[Chapter~3, \S~4]{Dixmier77}), which happens if and only if $\operatorname{ker}(\widehat{\pi_m}) = \{0\}$, i.e., $\widehat{\pi_m}$ is faithful. 
\end{proof}

There are many interesting examples of semi-pre-$C^*$-algebras. 
In Examples~\ref{ex:GA} and~\ref{ex:semi-Roe} below, we will interpret the classes of group algebras and uniform Roe algebras
from the perspective of semi-pre-$C^*$-algebras (see, e.g., \cite{Sc09,Ozawa13} for additional examples). 

We consider first the class of pre-$C^*$-algebras that we will need in the following two sections.

\begin{definition}
A unital pre-$C^*$-algebra $\cA$ is a unital $*$-algebra endowed with a norm $\|\cdot\|_{\cA}$ satisfying all the properties of a 
$C^*$-norm except possibly completeness.
\end{definition}

\begin{example}\label{basis}
Unital pre-$C^*$-algebras can be seen as special semi-pre-$C^*$-algebras as follows. Let $\cA$ be a unital pre-$C^*$-algebra and 
let $\iota\colon \cA\to \cB$ be the canonical injective $*$-homomorphism of $\cA$ into its completion $\cB$. Then $\cA_+=\cA\cap\cB_+$
is an archimedean closed $*$-positive cone making $\cA$ a semi-pre-$C^*$-algebra, with zero infinitesimal ideal. (Note that the archimedian closeness
follows from the fact that $\cB_+$ is norm-closed.) Conversely, if $\cA_+$ is an archimedian closed $*$-positive cone with zero infinitesimal ideal,
then the embedding $\iota\colon \cA\to C^*_\mathrm{univ}(\cA)$ gives a norm on $\cA$, and $\cA$ is a unital pre-$C^*$-algebra with completion 
$C^*_\mathrm{univ}(\cA)$ and with $\cA_+=C^*_\mathrm{univ}(\cA)_+\cap \cA.$
In other words, pre-$C^*$-algebras are the same as semi-pre-$C^*$-algebras whose $*$-positive cones are archimedian and have zero
infinitesimal ideal.
\end{example}

\begin{remark}\label{rem:nonessential}
Note that, in general, even if a positive $*$-representation $\pi\colon \cA\to \cL(\cH)$ is essential (i.e., $\pi(\cA)$ contains
no nonzero compact operators), the corresponding representation $\widehat{\pi}\colon C^*_\mathrm{univ}(\cA)\rightarrow \cL(\cH)$
need not be essential. In fact, consider the pre-$C^*$-algebra of polynomials in $C(\{-1\}\cup [0,1])$ acting by multiplication
on the Hilbert space $\cH=\C\oplus L^2(0,1)$. A nonzero polynomial is always nonzero on $[0,1]$ and hence it is represented
by a non-compact multiplication operator. Nevertheless, the completion $C(\{-1\}\cup [0,1])$ does contain a rank-one projection
given by the characteristic function of $\{-1\}$.
\end{remark}

We conclude this section recalling the notion of completely positive maps on semi-pre-$C^*$-algebras. Let $M_n(\cA)$ be the matrix algebra 
over a semi-pre-$C^*$-algebra $\cA$, which is a semi-pre-$C^*$-algebra with respect to the following $*$-positive cone
\begin{align*}
M_n(\cA)_+=\left\{\left(\sum_{k=1}^m (X_{k,i})^*A_kX_{k,j}\right)_{i,j}: m\in \N, A_k\in \cA_+, (X_{k,i})_{k,i}\in M_{m,n}(\cA)\right\}\;.
\end{align*}
It can be shown that $C^*_\mathrm{univ}(M_n(\cA))=M_n(C^*_\mathrm{univ}(\cA))$ for every $n\in \N$ (see \cite[Lemma~10]{Ozawa13}).

A linear map $\phi\colon\cA\rightarrow \B$ between semi-pre-$C^*$-algebras is \emph{completely positive} (c.p.) if 
$\text{id}\otimes \phi\colon M_n(\cA)\rightarrow M_n(\B)$ is positive for every $n$. If, in addition, $\phi$ is unital it is called
\emph{unital completely positive} (u.c.p.).
It is clear that states and (unital) positive $*$-homomorphisms are u.c.p. maps.

\begin{theorem}[{Stinespring dilation theorem, \cite[Theorem~12]{Ozawa13}}]\label{thm:pre-stine}
Let $\cA$ be a semi-pre-$C^*$-algebra and $\phi\colon\cA\rightarrow \cL(\cH)$ be a c.p. map, then there are a Hilbert space $\widehat{\cH}$, a
positive $*$-representation $\pi\colon\cA\rightarrow \cL(\widehat{\cH})$, and an operator $V\in \cL(\cH,\widehat{\cH})$ such that
\begin{align*}
\phi(A)=V^*\pi(A)V
\end{align*}
for $A\in \cA$. In particular, if $\phi$ is unital, then $V$ is an isometry.
\end{theorem}

\section{F\o lner conditions for semi-pre-$C^*$-algebras}\label{sec:Foelner-star}

In this section we will start considering F\o lner-type conditions in the context of Hilbert space operators
and other classes of $*$-algebras closely related to $C^*$-algebras. More precisely, we will
generalize results stated for F\o lner $C^*$-algebras in Section~4 of \cite{AL14} to the context of 
semi-pre-$C^*$-algebras. 

\subsection{F\o lner semi-pre-$C^*$-algebras}\label{subsec:foelner}
Now we turn to the discussion of F\o lner conditions for operators. This notion, and its relation to amenable tracial states, has
its origin in Connes' seminal paper on the classification of injective factors \cite{Connes76,ConnesIn76}. The following definition is an operator theoretic analogue of the previous
Definitions~\ref{def:metric-amenability} and~\ref{def:alg-amenable2}, where the r\^{o}le of F\o lner nets is played by nets of nonzero finite-rank projections.

\begin{definition}\label{def:Foelner2}
Let $\mathcal{T}\subset\cL(\cH)$ be a collection of linear and bounded
operators on a complex separable Hilbert space 
$\mathcal{H}$.ç
 \begin{itemize}
  \item[(i)] A net of nonzero finite rank orthogonal projections $\{P_i\}_{i\in I}\subset\cL(\cH)$ is called a {\em F{\o}lner net for $\mathcal{T}$} if
  \begin{equation}\label{eq:F1-2}
   \lim_{i} \frac{\|T P_i-P_i T\|_2}{\|P_i\|_2} = 0\;\;,\quad \text{for all } T\in\mathcal{T} \;,
  \end{equation}
  where $\|\cdot\|_2$ denotes the Hilbert-Schmidt norm. We call $\{P_i\}_{i\in I}$ a \emph{proper} F{\o}lner net if, in addition, 
  it converges to $\1 \in \cL(\cH)$ in the strong operator topology.
  \item[(ii)] For a finite set $\cF\subset\mathcal{T}$ and $\varepsilon \geq 0$, an \emph{$(\cF, \varepsilon)$-F\o lner projection} is a nonzero finite 
  rank orthogonal projection $P \in \cL(\cH)$ such that
  \begin{equation}\label{eq:FC2}
   \frac{\|T P-P T\|_2}{\|P\|_2} \leq \varepsilon \;\;,\quad \text{for all } T\in\cF \;.
  \end{equation}
  \item[(iii)] $\mathcal{T}$ satisfies the {\em F\o lner condition} if for any finite set $\cF\subset\mathcal{T}$ and any $\varepsilon >0$, 
  there exists an $(\cF, \varepsilon)$-F\o lner projection.
  \item[(iv)] $\mathcal{T}$ satisfies the {\em proper} F\o lner condition if for any finite set $\cF\subset\mathcal{T}$, any finite rank orthogonal
  projection $Q \in \cL(\cH)$, and any $\varepsilon >0$, there exists an $(\cF, \varepsilon)$-F\o lner projection $P \in \cL(\cH)$ such that $P \geq Q$.
\end{itemize}
\end{definition}

\begin{remark}\label{rem:trace-class}
      It is easy to see that the (proper) F\o lner condition is equivalent to having a (proper) F{\o}lner net.  
      Note that if $\mathcal{T}\subset\cL(\cH)$ has a F\o lner net $\{P_i\}_{i\in I}$, then it is also a F\o lner net for 
      $C^*(\cT,\1)$ (the $C^*$-algebra generated by $\cT$ and $\1$). See \cite{AL14} for additional results. 
\end{remark}

\begin{lemma}\label{lem:proper-Foelner-rank}
A collection $\mathcal{T}$ of linear and bounded operators in $\cL(\cH)$ satisfies the proper F\o lner condition if and only if for any finite set $\cF\subset\mathcal{T}$ and any $\varepsilon, N >0$, there exists an $(\cF, \varepsilon)$-F\o lner projection $P \in \cL(\cH)$ such that $\| P \|_2 \geq N$.
\end{lemma}

\begin{proof}
 The ``only if'' part is evident. For the ``if'' part, we shall directly verify the definition of the proper F\o lner condition. Given any finite set $\cF\subset\mathcal{T}$, any finite rank orthogonal projection $Q \in \cL(\cH)$, and any $\varepsilon >0$, we put 
 \[
  N = \max_{T \in \cF} \frac{4 \|T\| \|Q\|_2 }{ \varepsilon}
 \]
 and use the assumption to find an $(\cF, \frac{\varepsilon}{2} )$-F\o lner projection $P_0 \in \cL(\cH)$ 
 such that $\| P_0 \|_2 \geq N$. Let $P = P_0 \vee Q$ and $P^\perp = P - P_0$. Then we have $\|P\|_2 \geq \|P_0\|_2 \geq N$ 
 and $\|P^\perp\|_2 \leq \|Q\|_2$. Thus for any $T\in\mathcal{T}$, we have
 \begin{align*}
   \frac{\|T P-P T\|_2}{\|P\|_2} \leq \frac{\|T P_0 - P_0 T\|_2}{\|P\|_2} + \frac{\|T P^\perp - P^\perp T\|_2}{\|P\|_2} \leq \frac{\varepsilon}{2}  + \frac{2 \|T\| \|P^\perp\|_2}{\|P\|_2} = \varepsilon \; .
  \end{align*}
  Hence $P$ is an $(\cF, \varepsilon)$-F\o lner projection $P \in \cL(\cH)$ such that $P \geq Q$. 
\end{proof}

It is worth mentioning that the existence of a F\o lner net for a set of operators $\mathcal{T}$ is a
weaker notion than quasidiagonality. Recall that a separable set of operators
$\mathcal{T}\subset\mathcal{L}(\mathcal{H})$ is called 
quasidiagonal if there exists an increasing sequence of finite-rank
projections $\{P_n\}_{n\in \N}$ converging strongly to $\1$
and such that
\begin{equation}\label{QD}
\lim_{n}\|T P_n-P_n T\|=0\;,\quad \text{for all } T\in\mathcal{T} \;.
 \end{equation}
(See, e.g., \cite{Halmos70,Voiculescu93} or Chapter~16 in \cite{bBrown08} for additional motivation and results.)

Voiculescu's seminal article \cite{voicu91} gives an abstract characterization 
of quasidiagonality for unital separable $C^*$-algebras in terms of unital completely 
positive (u.c.p.) maps into matrices. This has become by now the
standard definition of quasidiagonality for operator algebras (see,
for example, \cite[definition~7.1.1]{bBrown08}). In \cite{AL14} we introduced the class
of unital and separable F\o lner $C^*$-algebras in terms of a sequence of u.c.p. maps into
matrices that are asymptotically multiplicative with respect to a weaker norm than the
operator norm. (This class was called {\em weakly hypertracial} in \cite{Bedos95}).
We will generalize this class of $C^*$-algebras to the context of 
semi-pre-$C^*$-algebras and prove several characterizations. Our results also show
that the separability assumption in \cite{AL14} is not essential.

\begin{definition}\label{def:FA2}
A unital semi-pre-$C^*$-algebra $\cA$ is {\em F\o lner} if there exists a net of u.c.p. maps
$\varphi_i\colon\cA\to M_{k(i)}(\C)$ which is asymptotically multiplicative, i.e.,
\begin{equation}\label{eq:assym-mult}
  \lim_i\|\varphi_i(AB)-\varphi_i(A)\varphi_i(B)\|_{2,\mathrm{tr}}=0\;,\quad \text{for all } A,B\in\cA\;,
\end{equation}
where $\|F\|_{2,\mathrm{tr}}\colon=\sqrt{\mathrm{tr}(F^*F)}$, $F\in M_{n}(\C)$ and 
$\mathrm{tr}(\cdot)$ denotes the unique tracial state on a matrix algebra
$M_{n}(\C)$. If, in particular, $\cA$ is a unital (pre-) $C^*$-algebra satisfying (\ref{eq:assym-mult})
we call it F\o lner (pre-) $C^*$-algebra.
\end{definition}

The class of F\o lner $C^*$-algebras have several properties. 
E.g., if a nonzero quotient of a unital $C^*$-algebra $\cA$ is a F\o lner $C^*$-algebra,
then $\cA$ itself is a F\o lner $C^*$-algebra (see Corollary~\ref{cor:nonzeroquotients} for
details). Moreover, we can also characterize F\o lner reduced crossed products. The
proof of the following result follows from Proposition~2.12 in
\cite{Bedos95}.

\begin{proposition}\label{pro:crossedprods} 
Let $\Gamma $ be a countable discrete group
and let $\alpha $ be an action of $\Gamma $ on a unital $C^*$-algebra $\cA$. Then the following conditions are equivalent:
\begin{enumerate}
\item $\cA \rtimes _{\alpha,r}\Gamma $ is a F\o lner $C^*$-algebra.
\item $\Gamma $ is amenable and $\cA$ has a $\Gamma
$-invariant amenable trace.
\item $\cA$ is a F\o lner $C^*$-algebra and $\Gamma $ is an
amenable group.
\end{enumerate}
\end{proposition}

\begin{example}\label{ex:Rosenberg}
Let $\Gamma$ be a countable discrete group. Due to the recent proof of Rosenberg's conjecture in \cite{TWW15}, we have the following equivalent conditions
(see also Sections~2.6 and 7.1 in \cite{bBrown08} and Proposition~\ref{pro:crossedprods} above):
\begin{enumerate}
	\item $\Gamma$ is amenable;
	\item $C^*_r(\Gamma)$ is quasidiagonal; 
	\item $C^*_r(\Gamma)$ is nuclear;
	\item $C^*_r(\Gamma)$ is a F\o lner $C^*$-algebra.
\end{enumerate}

\end{example}

Below we will present several characterizations of F\o lner semi-pre-$C^*$-algebras. In particular, we will
establish the relation to amenable traces, introduced below. We use the name trace
meaning a tracial state, i.e., a normalized positive linear functional having the tracial property.

\begin{definition}\label{def:amenable}
Let $\cA\subseteq \cL(\cH)$ be a unital pre-$C^*$-algebra with the $*$-positive cone 
$\cA\cap \cL(\cH)_+$. A state $\tau$ on $\cA$ is called an \emph{amenable trace} if there exists 
a state $\psi$ on $\cL(\cH)$ which extends $\tau$ and has $\mathcal{A}$ in its centralizer, i.e., 
\[  
\psi(TA)=\psi(AT)\;,\quad \text{for all } T\in \cL(\cH) \text{ and } A\in \cA\;.
\]
The state $\psi$ is called a \emph{hypertrace} for $\cA$.
\end{definition}

If the pre-$C^*$-algebra $\cA$ in the above definition is in fact a $C^*$-algebra, then the definition of amenable 
trace does not depend on the choice of embedding $\cA\subseteq \cL(\cH)$ (see \cite[Proposition 6.2.2]{Ozawa13}). 
Since every $C^*$-algebra is spanned by its unitaries, we see that a pre-$C^*$-algebra $\cA$ admits an amenable 
trace if and only if the norm closure $\bar{\cA}$ in $\cL(\cH)$ admits an amenable trace in the sense of \cite[Definition 6.2.1]{Ozawa13}.

Below is our first main result in this section, which gives a variety of characterizations for F\o lner semi-pre-$C^*$-algebras. Again, we assume that all representations involved are unital.

\begin{theorem} \label{theorem:characFolner2}
Let $\cA$ be a semi-pre-$C^*$-algebra and $C^*_\mathrm{univ}(\cA)$ its universal $C^*$-algebra which we assume
to be nonzero. Then the following conditions are equivalent:
\begin{enumerate}
  \item\label{theorem:characFolner2-nonzero-Foelner} There exists a nonzero positive $*$-representation $\pi\colon \cA\rightarrow\cL(\cH)$ such that 
  $\pi(\cA)$ has a F\o lner net.
  \item\label{theorem:characFolner2-nonzero-trace} There exists a nonzero positive $*$-representation $\pi\colon \cA\rightarrow\cL(\cH)$ such that $\pi(\cA)$ has an amenable trace.
  \item\label{theorem:characFolner2-maximal-trace} Every maximal $*$-representation $\pi_m\colon \cA\rightarrow\cL(\cH_m)$ 
  (cf., Definition~\ref{def:maximal})
             satisfies that $\pi_m(\cA)$ has an amenable trace.
  \item\label{theorem:characFolner2-maximal-Foelner} There exists a maximal and essential $*$-representation $\pi_m\colon \cA\rightarrow\cL(\cH_m)$ 
             such that $\pi_m(\cA)$ has a proper F\o lner net. 
  \item\label{theorem:characFolner2-exists-faithful-Foelner} There exists a faithful and essential $*$-representation $\widehat{\pi} \colon C^*_\mathrm{univ}(\cA)\rightarrow\cL(\cH)$ 
             such that $\widehat{\pi} ( C^*_\mathrm{univ}(\cA)) $ has a proper F\o lner net.
  \item\label{theorem:characFolner2-forall-faithful-Foelner} Every faithful and essential $*$-representation $\widehat{\pi} \colon C^*_\mathrm{univ}(\cA)\rightarrow\cL(\cH)$ 
             satisfies that \\ 
             $\widehat{\pi} ( C^*_\mathrm{univ}(\cA)) $ has a proper F\o lner net.
  \item\label{theorem:characFolner2-net} There exists a net of triples $\{(\cH_i, \pi_i, P_i)\}_{i\in I} $, where each $\pi_i$ is a nonzero positive $*$-representation of 
             $\cA$ on the Hilbert space $\cH_i $ and $P_i \in \cL(\cH_i)$ is a finite rank nonzero orthogonal projection, such that 
 \begin{equation}\label{eq:assymp}
  \lim_{i} \frac{\| \pi_i (A) P_i-P_i \pi_i (A) \|_{2}}{\|P_i\|_{2} } = 0\;\;,\quad \text{for all } A\in\cA \;,
 \end{equation}
 where $\|\cdot\|_{2}$ denotes the Hilbert-Schmidt norm in $\cL(\cH_i)$.
\item\label{theorem:characFolner2-definition} $\cA$ is a F\o lner semi-pre-$C^*$-algebra.
\item\label{theorem:characFolner2-universal} $C^*_\mathrm{univ}(\cA)$ is a F\o lner $C^*$-algebra.
 \end{enumerate}
\end{theorem}
\begin{proof}
We show (1) $\Rightarrow$ (2) $\Rightarrow$ (3) $\Rightarrow$ (4) $\Rightarrow$ (5) $\Rightarrow$ (6) $\Rightarrow$ (7) $\Rightarrow$ (1) 
and (7) $\Leftrightarrow$ (8) $\Leftrightarrow$ (9).

(1) $\Rightarrow$ (2): Since $\pi(\cA)$ has a F\o lner net, then its norm closure $\overline{\pi(\cA)}$ has the same 
F\o lner net and by Theorem~1.1 in \cite{Bedos95} we conclude that $\overline{\pi(\cA)}$ (hence also $\pi(\cA)$) has an
amenable trace.

(2) $\Rightarrow$ (3): Let $\rho\colon  \cA\rightarrow \cL(\cK)$ be a nonzero positive $*$-representation such that there exists a hypertrace
$\psi$ for $\rho(\cA)$ and, therefore, also for $\overline{\rho(\cA)}$. Moreover, let
$\pi_m\colon \cA\rightarrow \cL(\cH_m)$ be a maximal $*$-representation (cf., Definition~\ref{def:maximal}) and note that, by definition, we have
\[
 \|\rho(A)\|\leq \|\pi_m(A)\|\;,\quad \text{for all } A\in\cA\;.
\]
Therefore, the $*$-homomorphism $\Phi\colon \pi_m(\cA)\rightarrow \rho(\cA)$ given
by $\Phi(\pi_m(A))=\rho(A)$ can be extended to a $*$-homomorphism on the corresponding
closures $\Phi\colon \overline{\pi_m(\cA)}\rightarrow\overline{\rho(\cA)}$.
By Arveson's extension theorem, there is a u.c.p. map $\widehat{\Phi}\colon \cL(\cH_m) \rightarrow \cL(\cK)$ which
extends $\Phi$. It remains to show the state $\phi:=\psi\circ \widehat{\Phi}$ is a hypertrace for $\overline{\pi_m(\cA)}$.
Since the $C^*$-algebra $\overline{\pi_m(\cA)}$ is contained in the multiplicative domain of $\widehat{\Phi}$, we have
for each $A\in \overline{\pi_m(\cA)}$, $T\in \cL(\cH_m)$ that
\begin{align*}
\phi(AT)=\psi\left(\Phi(A)\widehat{\Phi}(T)\right)=\psi\left(\widehat{\Phi}(T)\Phi(A)\right)=\phi(TA)\;.
\end{align*}
This shows that $\phi$ is a hypertrace for $\overline{\pi_m(\cA)}$.

(3) $\Rightarrow$ (4): Let $\pi_{m_0}\colon \cA \rightarrow \cL(\cH_{m_0})$ be a maximal $*$-representation such that
$\pi_{m_0}(\cA)$ has an amenable trace. Consider the infinite sum of representations $\pi_{m_0}$ which gives a 
nonzero positive essential $*$-representation
\begin{align*}
\pi_m:=\mathop{\oplus}^\infty \pi_{m_0}\colon \cA\rightarrow \cL\Big(\mathop{\oplus}^\infty \cH_{m_0}\Big)\;,
\end{align*}
which, by construction, is also maximal. According to Definition~\ref{def:Foelner2}~(iv) we have to show
that for any finite set $\cF\subseteq \cA$, any finite rank orthogonal projection $Q\in \cL\left(\mathop{\oplus}\limits^\infty  \cH_{m_0}\right)$, 
and any $\varepsilon>0$, there exists a nonzero finite rank orthogonal projection $P\in \cL\left(\mathop{\oplus}\limits^\infty  \cH_{m_0}\right)$
such that $P\geq Q$ and
\begin{align*}
\frac{\|\pi_m(A)P-P\pi_m(A)\|_2}{\|P\|_2}\leq \varepsilon\;,\quad \text{for all } A\in \cF\;.
\end{align*}
Since $\overline{\pi_{m_0}(\cA)}$ has an amenable trace, then by Theorem~1.1 in \cite{Bedos95} 
there exists an $(\cF,\varepsilon/2)$-F\o lner projection $P_0\in\cL(\cH_{m_0})$. Choose $N\in \N$ such that 
\[ 
  N>\frac{4\|Q\|_2\cdot\|\pi_m(A)\|}{\varepsilon\cdot \|P_0\|_2} \;,\quad \text{for all } A\in \cF\;,
\]
and define the nonzero finite rank orthogonal projection 
\[ 
P:=Q\vee \left(\mathop{\oplus}^N P_0\right) \in \cL\Big(\mathop{\oplus}^\infty  \cH_{m_0} \Big) 
\]
which satisfies $P \geq Q$. Moreover, for all $A\in \cF$ we have
\begin{align*}
\frac{\|\pi_m(A)P-P\pi_m(A)\|_2}{\|P\|_2}\leq \frac{2\|Q\|_2\cdot\|\pi_m(A)\|}{N\cdot
\|P_0\|_2}+\frac{N\cdot\|\pi_{m_0}(A)P_0-P_0\pi_{m_0}(A)\|_2}{N\cdot \|P_0\|_2}\leq \varepsilon\;.
\end{align*}

(4) $\Rightarrow$ (5) Let $\pi_m\colon \cA\rightarrow\cL(\cH_m)$ be a maximal and essential $*$-representation
such that $\pi_m(\cA)$ has a proper F\o lner net. Then, $\widehat{\pi_m}$ is faithful by 
Lemma~\ref{lem:max-faithful} and since $\pi_m(\cA)=\widehat{\pi_m}(\iota(\cA))$ it follows that
$\widehat{\pi_m}(C^*_\mathrm{univ}(\cA))$ has the same proper F\o lner net. As mentioned in Remark~\ref{rem:nonessential},
$\widehat{\pi_m}$ need not be essential. But considering again the infinite direct sum of $\widehat{\pi_m}$
and adapting the F\o lner sequence one obtains the claim.

(5) $\Rightarrow$ (6) Let $\widehat{\pi_0} \colon C^*_\mathrm{univ}(\cA)\rightarrow\cL(\cH_0)$ be a faithful and essential 
$*$-representation such that $\widehat{\pi_0}(C^*_\mathrm{univ}(\cA))$ has a proper F\o lner net. For simplicity we will identify 
$C^*_\mathrm{univ}(\cA)$ with $\widehat{\pi_0}(C^*_\mathrm{univ}(\cA))\subset\cL(\cH_0)$. Consider now an arbitrary faithful and essential
$*$-representation $\widehat{\pi} \colon C^*_\mathrm{univ}(\cA)\rightarrow\cL(\cH)$. By Voiculescu's theorem (which is also true in the non-separable
context, cf., \cite[\S 1.7]{bBrown08}, \cite[Theorem~3.14]{Hadwin81}), the identity $*$-representation of $C^*_\mathrm{univ}(\cA)$ and $\widehat{\pi}$ are 
approximately unitarily equivalent. Therefore, for any $\varepsilon$ there is a unitary $U\colon \cH\to\cH_0$ such that
\[
 \|B-U\widehat{\pi}(B)U^{-1}\|<\frac{\varepsilon}{3}\;,\quad \text{for all } B\in C^*_\mathrm{univ}(\cA)\;.
\]

Fix $\varepsilon>0$, a finite set $\cF\subset C^*_\mathrm{univ}(\cA)$ and a finite rank orthogonal projection $R$ on $\cH$.
Then, since $C^*_\mathrm{univ}(\cA)(\equiv\widehat{\pi_0}(C^*_\mathrm{univ}(\cA)))$ has a proper F\o lner net, there is a
$(\cF,\frac{\varepsilon}{3})$-F\o lner projection $P_0$ on $\cH_0$ satisfying $P\geq U R U^{-1}$ (hence
$U^{-1}P_0U \geq R$). It remains to show that $Q :=U^{-1}P_0U $ is an $(\widehat{\pi}(\cF),\varepsilon)$-F\o lner 
projection on $\cH$. Now for any $B\in\cF$ we have
\begin{eqnarray*}
 \|\widehat{\pi}(B)\,Q-Q\,\widehat{\pi}(B)\|_2 & = & \| U\,\widehat{\pi}(B)\,U^{-1} \;P_0-P_0 \;U\, \widehat{\pi}(B)\,U^{-1} \|_2  \\[2mm]
                       &\leq & 2 \, \|U\widehat{\pi}(B)U^{-1}-B\|\, \|P_0\|_2 + \|B\,P_0-P_0\,B\|_2 \\[2mm]
                       &\leq & \varepsilon \, \|P_0\|_2           \;,
\end{eqnarray*}
which shows our claim.

(6) $\Rightarrow$ (7) Let $\widehat{\pi} \colon C^*_\mathrm{univ}(\cA)\rightarrow\cL(\cH)$ be a faithful and essential 
$*$-representation such that $\widehat{\pi}(C^*_\mathrm{univ}(\cA))$ (hence also $\pi(\cA)$ has a proper F\o lner net $\{P_i\}_i$). 
Take the triples $\{(\cH, \pi, P_i)\}_{i\in I}$ which, by definition, satisfy the asymptotic relation in Eq.~(\ref{eq:assymp}).

(7) $\Rightarrow$ (1): Let $\{(\cH_i, \pi_i, P_i)\}_{i\in I}$ be the net of triples in condition (7). Consider the following nonzero positive
$*$-representation $\pi:=\oplus_i \pi_i\colon  \cA\rightarrow \cL(\oplus_i \cH_i)$ and nonzero finite rank orthogonal projections 
$Q_i:=(0,\ldots,0,P_i,0,0,\ldots) \in \cL(\oplus_i \cH_i)$. It is clear that $\{Q_i\}_{i\in I}$ is a F\o lner net for $\pi(\cA)$.

(8) $\Rightarrow$ (7): Assume that $\cA$ is a F\o lner semi-pre-$C^*$-algebra and denote by
$\varphi_i\colon\cA\to M_{k(i)}(\C)$ the u.c.p. maps satisfying
\[ 
  \lim_i\|\varphi_i(AB)-\varphi_i(A)\varphi_i(B)\|_{2,\mathrm{tr}}=0\;,\quad \text{for all } A,B\in\cA\;.
\]
From Stinespring dilation theorem for semi-pre-$C^*$-algebras (cf., Theorem~\ref{thm:pre-stine})
there exist Hilbert spaces $\cH_i$, nonzero (unital) positive $*$-representations $\pi_i\colon \cA\rightarrow \cL(\cH_i)$ and isometries $V_i\in
\cL(\C^{k(i)},\cH_i)$ such that $\varphi_i(A)=V_i^*\pi_i(A)V_i$ for $A\in \cA$. Each range projection $P_i:=V_iV^*_i\in \cL(\cH_i)$ is 
nonzero and has finite rank $k(i)$.

Since each $V_i$ is an isometry we have for any $F\in M_{k(i)}(\C)$, that 
$\|VF\|_{2,\mathrm{tr}}=\|F V^*\|_{2,\mathrm{tr}}=\|F\|_{2,\mathrm{tr}}$, and, therefore,
\begin{align*}
\|\varphi_i(A^*A)-\varphi_i(A^*)\varphi_i(A)\|_{2,\mathrm{tr}}&=\|P_i\pi_i(A^*A)V_i-P_i\pi_i(A^*)P_i\pi_i(A)V_i\|_{2,\mathrm{tr}}\\
&=\frac{\|P_i\pi_i(A^*)(1-P_i)\pi_i(A)P_i\|_2}{\|P_i\|_2}\;.
\end{align*}
The following computation completes the proof the implication:
\begin{align*}
\lefteqn{\frac{\| \pi_i (A) P_i-P_i \pi_i (A) \|_{2}}{\|P_i\|_{2}}} \qquad\qquad  &\\[2mm]
&\leq  \frac{\| (1-P_i)\pi_i(A)P_i \|_{2}}{\|P_i\|_{2}}+\frac{\| P_i\pi_i(A)(1-P_i) \|_{2}}{\|P_i\|_{2}}\\[2mm]
&\leq \left(\frac{\|P_i\pi_i(A^*)(1-P_i)\pi_i(A)P_i\|_2}{\|P_i\|_2}\right)^{1/2}+
      \left(\frac{\|P_i\pi_i(A)(1-P_i)\pi_i(A^*)P_i\|_2}{\|P_i\|_2}\right)^{1/2} \\[2mm]
&=\Big(\|\varphi_i(A^*A)-\varphi_i(A^*)\varphi_i(A)\|_{2,\mathrm{tr}}\Big)^{1/2}+\Big(\|\varphi_i(AA^*)-\varphi_i(A)\varphi_i(A^*)\|_{2,\mathrm{tr}
}\Big)^{1/2} \\[2mm]
&\rightarrow 0\;,
\end{align*}
where the second inequality follows from the Cauchy-Schwarz inequality.

(7) $\Rightarrow$ (8): Consider a net of triples $\{(\cH_i, \pi_i, P_i)\}_{i\in I} $, where each $\pi_i\colon\cA\to\cL(\cH_i)$ is a nonzero
(unital) positive $*$-representations and $P_i \in \cL(\cH_i)$ are nonzero finite rank orthogonal projections, such that 
 \begin{align*}
  \lim_{i} \frac{\| \pi_i (A) P_i-P_i \pi_i (A) \|_{2}}{\|P_i\|_{2} } = 0\;\;,\quad \text{for all } A\in\cA \;,
 \end{align*}
 where $\|P_i\|_{2}$ denotes the Hilbert-Schmidt norm in $\cL(\cH_i)$. Define a net of u.c.p. maps $\varphi_i\colon \cA\rightarrow M_{k(i)}(\C)$ 
 by $\varphi_i(A)=P_i\pi_i(A)P_i$, where $k(i)$ denotes the rank of $P_i$. Moreover, for each $A,B\in \cA$ we have
\begin{align*}
\|\varphi_i(AB)-\varphi_i(A)\varphi_i(B)\|_{2,\mathrm{tr}}&=\|P_i(P_i\pi_i(A)-\pi_i(A)P_i)\pi_i(B)P_i\|_{2,\mathrm{tr}}  &\\[2mm]
&=\frac{\|P_i(\pi_i(A)-\pi_i(A)P_i)\pi_i(B)P_i\|_2}{\|P_i\|_{2}}&\\[2mm]
&\leq \frac{\|P_i\pi_i(A)-\pi_i(A)P_i\|_2}{\|P_i\|_{2}}\|\pi_i(B)\| &\\[2mm]
&\leq \frac{\|P_i\pi_i(A)-\pi_i(A)P_i\|_2}{\|P_i\|_{2}} \|\pi_m(B)\| \\[2mm]
&\rightarrow 0\;,
\end{align*}
where $\pi_m$ is a maximal $*$-representation of $\cA$ which exists by (4).

(8) $\Leftrightarrow$ (9): If $\cA$ is a F\o lner semi-pre-$C^*$-algebra, then, by (5), there is a faithful and essential 
$*$-representation $\widehat{\pi} \colon C^*_\mathrm{univ}(\cA)\rightarrow\cL(\cH)$ such that $\widehat{\pi}(C^*_\mathrm{univ}(\cA))$ has 
a proper F\o lner net and by Theorem~1.1 in \cite{Bedos95} also an amenable trace.
The implication ($\Rightarrow$) follows then by Theorem~6.2.7 in \cite{bBrown08}. The reverse implication 
is obvious.
\end{proof}

\begin{remark}
The preceding result shows, in the spirit of Theorem~4.1 in \cite{AL14}, different perspectives of the representation-independent notion of F\o lner semi-pre-$C^*$-algebra.
Note that if the semi-pre-$C^*$-algebra $\cA$ is a separable $C^*$-algebra, then all maximal (in particular, the universal)
positive $*$-representations $\pi_m$ on $\cA$ are faithful. Hence we recover Theorem~4.3 in \cite{AL14}
as a special case. The proof of the implication (2) $\Rightarrow$ (3) also shows that the notion of amenable trace for 
semi-pre-$C^*$-algebras does not depend on the maximal embedding.
\end{remark}

As a consequence of Theorem~\ref{theorem:characFolner2}~(ii) we have 
the following result in the nonseparable context (cf., \cite[Corollary~4.5]{AL14}).

\begin{corollary}\label{cor:nonzeroquotients} 
Let $\cA$ be a unital $C^*$-algebra. 
If a nonzero quotient of $\cA$ is a F\o lner $C^*$-algebra,
then $\cA$ is a F\o lner $C^*$-algebra. In particular, any $C^*$-algebra
admitting a finite-dimensional representation and the maximal
group $C^*$-algebra of a discrete group are examples of F\o lner $C^*$-algebras.
\end{corollary}

\begin{example}\label{ex:GA}
Semi-pre-$C^*$-algebras provide a useful framework to understand the algebraic aspects of
group $C^*$-algebras. In particular, one can use this approach to distinguish at the algebraic level the maximal and the reduced group $C^*$-algebras. Following \cite[Example~1]{Ozawa13}, 
we will analyze below these algebras as F\o lner semi-pre-$C^*$-algebras.

Let $\Gamma$ be a discrete group and denote by $\cA:=\C \Gamma$ its group algebra over $\C$. 
Then $\cA$ is a semi-pre-$C^*$-algebra with respect to the algebraic $*$-positive cone $\cA_{+, \operatorname{alg}}:=\C \Gamma_+$. 
It is clear from the definitions that $C^*_\mathrm{univ}(\cA)=C^*(\Gamma)$, the maximal group $C^*$-algebra of $\Gamma$. 
Since $C^*(\Gamma)$ has always a character (i.e., a one-dimensional representation) it follows from
Corollary~\ref{cor:nonzeroquotients} that it is a F\o lner $C^*$-algebra. Therefore, 
Theorem~\ref{theorem:characFolner2} implies that $(\cA,\cA_{+, \operatorname{alg}})$ is a F\o lner semi-pre-$C^*$-algebra.

If we equip the group algebra $\cA:=\C \Gamma$ with the following larger archi\-me\-dean closed $*$-positive cone
\begin{align*}
\cA_{+, \operatorname{r}}=\{ f\in\cA  : \text{there exists a net}\; \{f_n\}_n\subset\cA_{+, \operatorname{alg}}\;\text{such that}\; f_n\to f\;\text{ pointwise } \}\;,
\end{align*}
then $(\cA,\cA_{+, \operatorname{r}})$ is a semi-pre-$C^*$-algebra and $C^*_\mathrm{univ}(\cA)=C_r^*(\Gamma)$, the reduced group $C^*$-algebra of $\Gamma$. 
Since $C_r^*(\Gamma)$ is a F\o lner $C^*$-algebra if and only if $\Gamma$ is amenable we conclude that $(\cA,\cA_{+, \operatorname{r}})$ is a 
F\o lner semi-pre-$C^*$-algebra if and only if $\Gamma$ is amenable and this is the case if and only if $\cA_{+, \operatorname{r}}=\text{arch}(\cA_{+, \operatorname{alg}})$ 
(see \cite[Example~1]{Ozawa13}). 
Finally, we note that it may happen that the algebraic $*$-positive cone $\C \Gamma_+$ is already archi\-me\-dean closed, which is the case for 
$\Gamma=\mathbb{F}_d$ (\cite{Netzer-Thom-13}) and $\Gamma=\Z^2$ (\cite{Scheiderer06}).
\end{example}

\begin{example}
One of the reasons for working in the context of semi-pre-$C^*$-algebras, as opposed to more general $*$-algebras, is the assumption
of Combes axiom in Definition~\ref{def:semi-pre}, which implies that
the u.c.p. maps are completely bounded. This is an indispensable requirement for Theorem~\ref{theorem:characFolner2}
above. The following example illustrates that, in particular, the implication
(8) $\Rightarrow$ (1) is not true in a more general setting: consider the locally compact Hausdorff 
space $X= [0,\infty)$ and for $n\in\N_0$ define $\cA_n=C(X, M_{2^n})$, the unital $*$-algebra of continuous functions on $X$ with values in the $*$-algebra of complex $2^n \times 2^n$-matrices. Let $\gamma_n \colon \cA_n \to \cA_{n+1}$ be the map defined by the following block diagonal matrix
\[
\gamma_n (f)= 
\begin{pmatrix}
 f & 0 \\
 0 & f(n+1)
\end{pmatrix}
\;,
\] 
where $f(n+1)$ is identified with the constant function on $X$ taking the value $f(n+1)$. 
Consider the algebraic inductive limit $\cA= \lim_n  \cA_n$, which is a unital $*$-algebra. 
We show next that there is a sequence $\varphi_n \colon \cA\to M_{2^n}$ 
of asymptotically multiplicative u.c.p. maps.
Given $n$, and any $m\ge n$, denote by $\psi_{m,n}$ the u.c.p. map $M_{2^m}\to M_{2^n}$ obtained by compressing a 
$2^m\times 2^m$ matrix into its $2^n\times 2^n$ upper left corner. Consider the map 
$\varphi _{m,n}\colon \cA_m \to M_{2^n}$ given by 
$$\varphi_{m,n}(f) = \psi_{m,n} (f(0)).$$
The family of u.c.p. maps $\{\varphi_{m,n}: m\ge n \}$ is compatible with the transition homomorphisms $\gamma _i$ in the sense that $\varphi_{m,n} = \varphi_{m+1,n} \circ \gamma_m$, so it induces a u.c.p. maps
$\varphi_n\colon \cA\to M_{2^n}$ which, by construction, are asymptotically multiplicative.

Now we show that that there is no nonzero unital $*$-homomorphism from $\cA$ to any $C^*$-algebra. Suppose that $\pi \colon \cA \to \cB$ is such a $*$-homomorphism. Consider the element $X$ in $\cA$ obtained as the image in the 
(algebraic) direct limit of the function $f\in \cA_0$ given by $f(t) = t$. 
We will show that $\| \pi (X) \| \ge n $ for all $n\ge 0 $. Let $\pi_n \colon \cA_n \to \cB$ be the $*$-homomorphism
obtained by $\pi_n = \pi \circ \gamma_{n,\infty}$, where $\gamma _{n,\infty}\colon \cA_n\to \cA$ is the natural map into the direct limit. 
Then by the simplicity of the matrix algebra $M_{2^n}$, considered as the unital $*$-subalgebra of $\cA_n$ consisting of constant functions on $X$, we know that $\pi_n (M_{2^n})$ is a unital $*$-subalgebra of $\cB$ isomorphic to $M_{2^n}$. Standard matrix manipulations then show that $\cB \cong \cB_0 \otimes_{\mathbb{C}} M_{2^n}$ where $\cB_0$ is the commutant of $\pi_n (M_{2^n})$ inside $\cB$, and moreover, $\pi_n$ can be decomposed as
\[
	\delta_n \otimes \text{id}\colon C(X)\otimes_{\mathbb{C}} M_{2^n}\to \cB_0\otimes_{\mathbb{C}} M_{2^n}
\]
for some unital $*$-homomorphism $\delta_n \colon C(X)\to \cB_0$. But now observe that $X_n = \gamma _{n-1}\circ \cdots \circ \gamma _0 (f)$ is a diagonal matrix with one diagonal entry equal to $f(n) = n \cdot \1_X$. Consequently 
  $$\| \pi (X) \| = \| \pi (\gamma_{n,\infty}(X_n))\| = \| \pi_n(X_n) \| \ge \| \delta_n (n \cdot \1_X)\| = \| n \cdot \1_{\cB_0}\| = n \, ,$$
showing our claim.
\end{example}

\subsection{Properly F\o lner pre-$C^*$-algebras and the relation to algebraic amenability}\label{subsec:prop-F-algebras}

In the present subsection we will introduce the notion of properly F\o lner pre-$C^*$-algebra\footnote{Note that the 
name proper F\o lner $C^*$-algebra was used in a different sense in \cite{AL14}!} and we will make contact with the
notion of (proper) algebraic amenability introduced in Subsection~\ref{subsec:alg-amenable3}.

\begin{definition}\label{def:properly-Foelner}
 A (unital) pre-$C^*$-algebra $\cA$ is \emph{properly F\o lner} if for any finite subset $\cF \subset \cA$ and any  
 $\varepsilon > 0$, there exist a positive and cyclic representation $\pi \colon \cA \to \cL(\cH)$ and an 
 $(\pi(\cF), \varepsilon)$-F\o lner projection $P$ in the sense of Definition~\ref{def:Foelner2}~(ii) such that 
 the linear map 
 \[ 
    \langle\cF\rangle\ni A \mapsto P \pi(A) \Omega \in \cH
 \] 
 is injective, where $\Omega$ is a cyclic vector for $\pi$ and $\langle\cF\rangle:=\mathrm{span}\cF$.
\end{definition}

\begin{remark}\label{rem:properly-Foelner-definition}
\begin{enumerate}
 \item Applying Theorem~\ref{theorem:characFolner2}(7), we see that a properly F\o lner pre-$C^*$-algebra must also be F\o lner. 
 We will show later that, contrary to the notion of {\it proper F\o lner} $C^*$-algebra introduced in \cite{AL14}, the notion of 
 properly F\o lner $C^*$-algebra introduced here is strictly stronger than the notion of F\o lner $C^*$-algebra 
 (see \cite[Definition 4.1(ii), Proposition 4.2]{AL14}). 
 \item  It is clear that a finite-dimensional $C^*$-algebra is always properly F\o lner, 
 since it carries a faithful cyclic representation on a finite-dimensional Hilbert space. 
\end{enumerate}
 \end{remark}

 The following result shows a crucial difference between properly F\o lner and F\o lner pre-$C^*$-algebras.
 
 \begin{theorem}\label{prop:properly-Foelner-direct-sum-B}
 Let $\cA$ and $\cB$ be unital pre-$C^*$-algebras. If $\cA \oplus \cB$ is properly F\o lner and $\cB$ is finite-dimensional, then
 $\cA$ is properly F\o lner. 
\end{theorem}
 \begin{proof}
Assume that $\cA \oplus \cB$ is properly F\o lner and let $\{B_1 , \ldots, B_m\}$ be a linear basis of $\cB$.
It suffices to give a proof in the case when $\cA$ is infinite-dimensional, since otherwise the conclusion is obvious by 
Remark~\ref{rem:properly-Foelner-definition}.

Let $\cF_0 \subset \cA $ and $\varepsilon > 0$ be given.  
To prove our assertion we need to find a cyclic representation $\pi_0\colon\cA\to\cL(\cH_0)$ with cyclic vector $\Omega_0$ and satisfying
the conditions of Definition~\ref{def:properly-Foelner}. Consider the finite set
$\cF = \cF_0 \sqcup \{B_1 , \ldots, B_m\} \subset \cA \oplus \cB$ (identifying $\cA$ with $\cA\oplus \{0\}$ and $\cB$ with 
$\{0\}\oplus\cB$). 
From Definition~\ref{def:properly-Foelner} applied to $\cA \oplus \cB$, there is a positive and cyclic representation 
$\pi \colon \cA \oplus \cB  \to \cL(\cH)$ with cyclic vector $\Omega$
and a $\left(\pi(\cF), \frac{\varepsilon}{2}\right)$-F\o lner projection $P$ in the sense of Definition~\ref{def:Foelner2}~(ii) such that 
the linear map $\mathrm{span}(\cF) \to \cH$ given by $C \mapsto P \pi (C) \Omega$ is injective.
The latter condition implies also that $\|P\|_1=\dim(P\cH)\geq\dim\mathrm{span}(\cF)$.
Note, in addition, that $P$ will also implement proper F\o lnerness for any smaller set
$\cF'\subset\cF$ and that the linear map $\pi(\cA \oplus \cB) \Omega \to P \cH$ given by $C \mapsto P \pi (C) \Omega$ has dense image and is thus surjective. Therefore, by adding more elements of small norms to $\cF$, without loss of generality, we may assume that $\|P\|_1 = \dim\mathrm{span}(\cF)$ and
 \[
  \frac{\varepsilon}{2}\; \frac{ \mathrm{dim}( \mathrm{span}(\cF) ) }{  \mathrm{dim}( \mathrm{span}(\cF) ) - m  } 
                      + \frac{4 m \max_{T \in \cF_0} \| T \| }{   \mathrm{dim}( \mathrm{span}(\cF) ) - m  } \leq \varepsilon \; .
 \]
For this estimate we need that $\dim\cB=m$ is finite. 
Next we show that the F\o lner projection $P$ is approximately well-adapted to the orthogonal structure induced on $\cH$ by the 
direct sum $\cA\oplus\cB$. We have $\cH=\cK_0\oplus\cK_1$, where
$\cK_0:=\overline{\pi(\cA\oplus \{0\})\Omega}$ and $\cK_1:=\overline{\pi(\{0\}\oplus\cB)\Omega}$. Moreover, $\mathrm{dim}\,\cK_1=m$ because the map
$C\mapsto \pi (C) \Omega $ is injective when restricted to $\mathcal B$.  
We denote by $R_0$ and $R_1$ the orthogonal projections onto $\cK_0$ and $\cK_1$, respectively.
It is clear that for any $A\in\cA$, $B\in\cB$, we have $R_0\pi(A\oplus B)R_0=\pi(A\oplus B)R_0$, hence $R_0$ reduces $\pi$ and we 
define $\pi_0:=\pi|_{\cK_0}\colon\cA\to\cL(\cK_0)$ which by construction is a cyclic representation with cyclic vector $\Omega_0:=\pi(\1\oplus 0)\Omega$.
Define the projections
\begin{equation}\label{eq:PQ}
 P_0:= P\vee R_1 -R_1\le R_0 \quad\mathrm{and}\quad Q:=P\vee R_1 -P \;.
\end{equation}
Then we have 
\[
\| P_0 \|_1 =   \| P\vee R_1 - R_1 \|_1 = \|  P- P\wedge R_1 \|_1 \ge \| P\|_1 -m   \ge \mathrm{dim}( \mathrm{span}(\cF) ) - m 
\]
and 
\[
\| Q \|_1 =   \| P\vee R_1 - P \|_1 = \|  R_1- P\wedge R_1 \|_1 \le m.   
\]
We now show that $P_0$ is a $(\pi(\cF_0),\varepsilon)$-F\o lner projection.
For any $T \in \cF_0$, we have, using Eq.~(\ref{eq:PQ}) as well as 
the trace-class norm to compute the size of the finite rank operators (see Remark~\ref{rem:trace-class}~(ii)),
 \begin{eqnarray*}
 \lefteqn{
  \frac{\| P_0\pi_0(T)  - \pi_0(T)P_0\|_1}{\|P_0\|_1}
  } &&\\
    &\leq & \ \frac{\| (P\vee R_1)  \pi(T\oplus 0) - \pi(T\oplus 0) (P\vee R_1)  \|_1}{\|P_0\|_1} 
                                              + \frac{\| R_1 \pi(T\oplus 0) - \pi(T\oplus 0)  R_1  \|_1}{\|P_0\|_1} \\
    &\leq & \ \frac{\| P  \pi(T\oplus 0) - \pi(T\oplus 0) P  \|_1}{\|P_0\|_1} 
                                              + \frac{\| Q \pi(T\oplus 0) - \pi(T\oplus 0)  Q  \|_1}{\|P_0\|_1} + 2\max_{T\in\cF_0}\|T\|\frac{m}{ \|P_0\|_1 }  \\
    &\leq & \ \frac{\varepsilon}{2} \frac{\|P\|_1 }{\|P_0\|_1} + 4\max_{T\in\cF_0}\|T\|\frac{m }{\|P_0\|_1} \\
    &\leq & \ \frac{\varepsilon}{2}\; 
                                                \frac{ \mathrm{dim}( \mathrm{span}(\cF) ) }{  \mathrm{dim}( \mathrm{span}(\cF) ) - m  } 
                                              + \frac{4 m \max_{T \in \cF_0} \| T \| }{   \mathrm{dim}( \mathrm{span}(\cF) ) - m  } \\
   &\leq &\ \varepsilon\;.
\end{eqnarray*}
Finally, we have that the map $\mathrm{span}(\cF _0)\to \mathcal K_0$ 
\[
 A\mapsto P_0\pi_0(A)\Omega_0=(P\vee R_1-R_1)R_0\pi(A\oplus 0)\Omega=(P\vee R_1)\pi(A\oplus 0)\Omega
\]
is injective. In summary, we have shown by verifying the conditions of Definition~\ref{def:properly-Foelner} that $\cA$ is properly F\o lner.
\end{proof}

\begin{remark}
 \label{rem:directsumOK}
 It is easy to show that $\mathcal A \oplus \mathcal B$ is properly F\o lner if $\mathcal A$ and $\mathcal B$ are. We will not need this result in the present paper. 
 \end{remark}

As an application of Theorem~\ref{theorem:characFolner2} we conclude by relating the concept of (proper) 
algebraic amenability given in Subsection~\ref{subsec:alg-amenable3} to the (proper) F\o lner property for pre-$C^*$-algebras.
Parts of the proof of the following result are inspired by that of \cite[Theorem 5.1]{Cec-Sam-08}.\footnote{In the proof 
of \cite[Theorem 5.1]{Cec-Sam-08}, the authors seem to use the non-stated condition that (in their notation) $\xi$ is a 
cyclic and {\it separating} vector for the action of $\cA$.} Nevertheless, in the following proof we exploit the
fact that the notion of F\o lner semi-pre-$C^*$-algebra presented in this article is intrinsic, i.e., representation-independent.

\begin{theorem}\label{theorem:alg-amen-to-Folner}
 Let $\cA$ be a (nonzero) unital pre-$C^*$-algebra, and $\cB$ be a dense subalgebra of $\cA$.
 \begin{enumerate}
  \item[(i)]  If $\cB$ is (left) algebraically amenable, then $\ol{\cA}$ is a F\o lner $C^*$-algebra. In particular, $\cA$ is a F\o lner pre-$C^*$-algebra.  
  \item[(ii)] If $\cB$ is (left) properly algebraically amenable, then $\ol{\cA}$ is a properly F\o lner $C^*$-algebra. In particular, $\cA$ is a properly F\o lner pre-$C^*$-algebra. 
  \end{enumerate}
  \end{theorem}

\begin{proof}
(i) We assume first that $\cA$ is countably dimensional as a complex vector space. Since
$\overline{\cA}$ is separable there is a faithful state $\phi$ on $\ol{\cA}$ (see \cite[3.7.2]{pedersen79}), and we denote by $\pi\colon\ol{\cA}\to \cL(\cH)$ the GNS representation
with cyclic vector $\Omega$.  By Theorem~\ref{theorem:characFolner2}~(1), it suffices to show that 
 given any finite set $\cF \subset \overline{\cA}$ and any $\varepsilon > 0$, there is a $(\pi(\cF), \varepsilon)$-F\o lner projection $P \in \cL(\cH)$.
 By our assumption, the algebra $\cB$ is dense in $\overline{\cA}$, whence we can find a finite set 
 $\cF'\subset\cB$ such that $\cF \cup \cF^* \subset N_{{\varepsilon}/{4}}(\cF')$, the $\frac{\varepsilon}{4}$-neighborhood of $\cF'$. 
 Now by the algebraic amenability 
 of $\cB$, we can find a nonzero finite dimensional linear subspace $W \subset \cB$ such that 
 \begin{equation}
  \label{eq:goodforB}
  \frac{\dim(B W +W)}{\dim(W)}\leq 1+ \left( \frac{\varepsilon}{ 4 \| B \| + 1 }  \right)^2 
 \end{equation}
for all $B \in\cF'$.

Consider now the linear map $\Psi\colon \overline{\cA} \to \cH$ given by
$ A \mapsto \pi (A)\Omega $, which is injective because of the faithfulness of $\phi$. Define $P$ to be the 
orthogonal projection onto the finite-dimensional subspace $\Psi(W) \subset \cH$. We claim that $P$ is a 
$(\pi(\cF), \varepsilon)$-F\o lner projection. Indeed, for any $ A \in \cF $, there are $B, B' \in \cF'$ such that 
$ \| A - B \| \leq \frac{\varepsilon}{4} $ and $ \| A^* - B' \| \leq \frac{\varepsilon}{4} $, and thus
 \begin{align*}
  \frac{\| \pi(A) P-P \pi(A) \|_2}{\|P\|_2}  = &\ \frac{ \| (\1 - P) \pi(A) P - P \pi(A) (\1 - P)  \|_2 }{\|P\|_2} \\
  \leq &\  \frac{ \| (\1 - P) \pi(A) P \|_2 + \| (\1 - P) \pi(A^*) P \|_2 }{\|P\|_2} \\
  \leq &\  \frac{ \| (\1 - P) \pi(B) P \|_2 + \| (\1 - P) \pi(B') P \|_2 }{\|P\|_2} \\
  &\ + \frac{ \| (\1 - P) \pi(A - B) P \|_2 + \| (\1 - P) \pi(A^* - B') P \|_2 }{\|P\|_2} \\
  \leq &\  \frac{ \| (\1 - P) \pi(B) P \|_2  }{\|P\|_2} + \frac{ \| (\1 - P) \pi(B') P \|_2 }{\|P\|_2} + \frac{ \varepsilon }{2} \; .
 \end{align*}
In order to estimate $ \| (\1 - P) \pi(B) P  \|_2 $, we choose an orthonormal basis $\{e_i\}_{i \in \{1, \cdots , \dim(B W +W) \} \sqcup I}$ 
for $\cH$, such that $\{e_1, \cdots, e_{\dim(W)}\}$ is an orthonormal basis for $\Psi (W)$, and $\{e_1, \cdots, e_{\dim(B W +W)}\}$ 
is an orthonormal basis for $\Psi (B W +W)$. Hence 
 \begin{align*}
   \| (\1 - P) \pi(B) P  \|_2 ^2 = &\  \sum_{i, j \in \{1, \cdots , \dim(B W +W) \} \sqcup I} 
   \Big| \big\langle (\1 - P) \pi(B) P  e_i , e_j \big\rangle \Big|^2 \\
   = &\  \sum_{ j = \dim(W) + 1}^{\dim(B W +W)} \left(  \sum_{i = 1}^{\dim(W)} | \langle \pi(B)e_i , e_j \rangle |^2 \right) \\
   \le &\  \sum_{ j = \dim(W) + 1}^{\dim(B W +W)}  \| \pi(B)^* e_j  \|^2 \\
   \le &\  \big( \dim(B W +W) - \dim(W) \big) \cdot  \| \pi(B) \|^2 \;.
 \end{align*}
 Applying the same method to $B'$ instead of $B$, we have 
 \[
  \| (\1 - P) \pi(B') P  \|_2 ^2 \le  \big( \dim(B' W +W) - \dim(W) \big) \cdot  \| \pi(B') \|^2
 \]
 as well. Noticing that $ \|P\|_2^2 = \dim(W) $, we finally have
 \begin{align*}
  & \ \frac{\| \pi(A) P-P \pi(A) \|_2}{\|P\|_2} \\
  \leq &\  \frac{ \| (\1 - P) \pi(B) P \|_2 }{\|P\|_2} + 
            \frac{ \| (\1 - P) \pi(B') P \|_2 }{\|P\|_2}  + \frac{ \varepsilon }{2} \\
  \leq &\    \sqrt{ \frac{ \dim(B W +W) - \dim(W) }{ \dim(W) } }  \| \pi(B) \| 
          +  \sqrt{ \frac{ \dim(B' W +W) - \dim(W) }{ \dim(W) } }  \| \pi(B') \| + \frac{ \varepsilon }{2} \\
     < &\  \frac{ \varepsilon }{4} + \frac{ \varepsilon }{4} + \frac{ \varepsilon }{2}  =  \varepsilon \;.
 \end{align*}
 This shows that $P$ is a $(\pi(\cF), \varepsilon)$-F\o lner projection.

 Now we show the general case. Let $\cF\subset \overline{\cA}$ be a finite subset and $\varepsilon$ a positive real number. By Theorem~\ref{theorem:characFolner2}~(7), 
 it suffices to show that there is a representation $\pi \colon \ol{\cA} \to \cL(\cH)$ and a $(\pi (\cF), \varepsilon)$-projection $P$. 
 There is a countably dimensional subalgebra
$\cB_0\subset\cB$ such that $\cF\cup \cF^* \subset \overline{\cB_0}$. By Proposition~\ref{pro:alg-amenability-countability}, we can choose a countably dimensional subalgebra $\cB_1$ 
of $\cB$ containing $\cB_0$ such
that $\cB_1$ is algebraically amenable. Now, let $\cA_1$ be a separable $C^*$-subalgebra of $\ol{\cA}$ containing $\cB_1$, 
choose a faithful state on 
$\cA_1$ and extend it to a state $\phi$ on $\ol{\cA}$. Let $\pi \colon \ol{\cA} \to \cL (\cH)$ be the GNS-representation associated to $\phi$. 
Proceeding as in the first step, we can find a 
  $(\pi(\cF), \varepsilon)$-F\o lner projection $P \in \cL(\cH)$, as desired.

  If $\ol{\cA}$ is a F\o lner $C^*$-algebra, then $\cA$ is a F\o lner pre-$C^*$-algebra, by Theorem~\ref{theorem:characFolner2}.

 (ii) We prove first the statement in the case when $\cA$ is a countably dimensional unital pre-$C^*$-algebra. Since
$\ol{\cA}$ is separable there is a faithful state $\phi$ on $\ol{\cA}$ and we denote by $\pi\colon \ol{\cA}\to \cL(\cH)$ the GNS representation
with cyclic vector $\Omega$. Consider as above the linear map $\Psi\colon\ol{\cA}\to\cH$ defined by $\Psi(A)=\pi(A)\Omega$. This is an isometry since $\phi$ is faithful, where we consider the $\| \cdot \|_2$-norm in $\ol{\cA}$ associated to the state $\phi$. 
According to Definition~\ref{def:properly-Foelner} it is enough
to show that for any finite $\cF\subset \ol{\cA}$ and any $\varepsilon>0$ there is a $(\pi(\cF),\varepsilon)$-F\o lner
projection $P$ on $\cH$ such that the linear map
\[
 \langle\cF\rangle :=\mathrm{span}\cF\ni B \mapsto P\pi(B)\Omega \subset\cH
\]
is injective. We may assume that $\cF = \{ f_1,\dots , f_n \}$ is a finite orthonormal subset of $(\ol{\cA}, \|\cdot \|_2 )$.  
We can find a finite set 
 $\cF'\subset\cB$ such that $\cF \cup \cF^* \subset N_{{\varepsilon}/{4}}(\cF')$. Moreover, there is a constant $\eta >0$ such that for any
 $n\times n$ scalar matrix $C= (c_{ij})$, if $|c_{ij} - \delta_{ij}| <\eta $ then $C$ is an invertible matrix. 
By observing that $\cB$ is dense in $\ol{\cA}$ in the $\|\cdot \|_2$-norm, we can find 
 an orthonormal subset $\cF ''= \{ f_1'',\dots , f_n'' \} \subset \cB$ such that $\| f_i -f_i'' \|_2 <\eta $. 

Since $\cB$ is properly algebraically amenable, there exists a $(\pi(\cF'),\varepsilon)$-F\o lner subspace $W$ of $\cB$ such that
\eqref{eq:goodforB} holds for all $B\in \cF '$,  and, in addition,  $\cF '' \subseteq W$.  
Let $P$ be the orthogonal projection onto $\Psi(W)$ and note that $P\Psi (A) =\Psi (A) $ for any 
$A\in \langle \cF'' \rangle$. As in the proof of (i), we conclude that $P$ is a
$(\pi(\cF),\varepsilon)$-F\o lner projection. Moreover, since $\| f_i -f_i'' \|_2 <\eta $ for all $i$, we have that 
$|\langle f_i,f_j'' \rangle - \delta_{ij}| < \eta $ for all $i,j$, and so the orthogonal projection onto $\langle \Psi (\cF '') \rangle $, 
$P''\colon \langle  \Psi (f_1) ,\dots, \Psi (f_n) \rangle\to \langle  \Psi (f_1''),\dots, \Psi (f_n'') \rangle$ is injective.  Hence
\[
 \langle\cF\rangle\ni B \mapsto P\pi(B)\Omega \subset\cH
\]
is injective too. This shows that $\ol{\cA}$ is properly F\o lner. 

The general case is reduced to the countably dimensional case just as in (i), using Proposition~\ref{pro:alg-amenability-countability}. It is clear from the definition that if $\ol{\cA}$ is properly F\o lner, then
$\cA$ is a properly F\o lner pre-$C^*$-algebra. 
 \end{proof}

\begin{remark}\label{rem:counter-ex}
 The converse of Theorem~\ref{theorem:alg-amen-to-Folner} is in general not true: a F\o lner pre-$C^*$-algebra 
 may contain a dense subalgebra not satisfying algebraic amenability. For example, let $\Gamma$ be a non-amenable group. 
 Then the group algebra $\C\Gamma$ is not algebraically amenable by Example~\ref{ex:group-algebra}. 
Nevertheless, $(\C \Gamma,\|\cdot\|_{\mathrm{max}})$ (or the maximal group $C^*$-algebra 
$C^*(\Gamma)$) is always a F\o lner pre-$C^*$-algebra, because of the representation induced by the 
group homomorphism from $\Gamma$ to the trivial group (see Corollary~\ref{cor:nonzeroquotients}). 
In contrast, $(\C \Gamma,\|\cdot\|_{\mathrm{r}})$ 
(or  its reduced group $C^*$-algebra $C^*_{r}(\Gamma)$) is not a F\o lner pre-$C^*$-algebra (cf., Proposition~\ref{pro:crossedprods}).
\end{remark}

\section{Amenability of uniform Roe algebras}\label{sec:alg-amenable2}

In Section~6 of \cite{ALLW-1} we use the example of the translation algebra to show the close relation
between metric space amenability and algebraic amenability. In the present section we add into this picture
the notion of F\o lner $C^*$-algebra.
What naturally connects these areas is a family of related constructions introduced by and named after Roe. 
For our purpose we will be focusing on various variants of the so-called uniform Roe algebras.

We briefly recall the definition of the uniform Roe algebra and its variants. Let $(X,d)$ be an extended metric space 
with bounded geometry (see Subsection~\ref{subsec:amenability-metric}). 
Every partial translation of $X$ gives rise to a partial isometry in $\cL(\ell^2(X))$ in the following way: 
if $t\colon A\rightarrow B$ is a partial translation (cf.~Definition~\ref{def:part-transl}), then define $V_t\in \cL(\ell^2(X))$ by
\begin{equation}\label{partial-trans}
{\langle { \delta_y} , V_t {\delta_x} \rangle} = (V_t)_{yx}:=
\begin{cases}
           1 & \text{if}\  y=t(x)\\
           0 & \text{otherwise} \;.
\end{cases}
\end{equation}
It is clear 
that $V_t$ is a partial isometry and that $V_t^*$ is the partial isometry associated to the partial translation $t^{-1}:B\rightarrow A$. 
Moreover, the support and range projections $V_t^* V_t$ and $V_t V_t^*$ are the characteristic functions of $A$ and $B$, respectively, considered as
multiplication operators in $\cL(\ell^2(X))$. The partial isometries mentioned before specify a Hilbert space representation of the 
translation algebra (with the coefficient field $\K=\C$) analyzed in \cite[Section~6]{ALLW-1}. 

\begin{definition}\label{def:transl alg/alge uniRoe}
The \emph{translation algebra} $\C_\mathrm{u}(X)$ is defined as the $*$-subalgebra in $\cL(\ell^2(X))$ generated by $V_t$ for all the partial
translations $t$ on $X$. The \emph{algebraic uniform Roe algebra} $C^*_\mathrm{u,alg}(X)$ is defined as the 
unital $*$-subalgebra in $\cL(\ell^2(X))$ generated by $\C_\mathrm{u}(X)$ and $\ell^\infty(X)$, where 
$\ell^\infty(X)$ is identified with diagonal operators in $\cL(\ell^2(X))$.
\end{definition}

We remark that the algebraic uniform Roe algebra is also often defined with the help of the notion of finite propagation (see Remark~\ref{rem:uniform-Roe-definition}).
Note that since $V_t V_s = V_{t \circ s}$ for any partial translations $t$ and $s$, it follows that $\C_\mathrm{u}(X)$ is linearly spanned by the generators $V_t$.

Also note that $\C_\mathrm{u}(X)$ is a dense subalgebra in $C^*_\mathrm{u,alg}(X)$ and that $\ell^\infty(X)\not\subset\C_\mathrm{u}(X)$ unless $X$ is finite. Indeed, $\ell^\infty(X) \cap \C_\mathrm{u}(X)$ is the dense subalgebra of $\ell^\infty(X)$ made up of all (complex) step functions, i.e., those having finite images. Nevertheless, the following lemma tells us that from a $C^*$-algebraic point of view, these two algebras do not make too much of a difference. 

\begin{lemma}\label{lem:same-representations}
Let $(X,d)$ be an extended metric space with bounded geometry. 
Any $*$-re\-pre\-sen\-tation of $\C_\mathrm{u}(X) $ can be extended uniquely to a $*$-representation of $C^*_\mathrm{u,alg}(X)$ up to unitary equivalence.
\end{lemma}

\begin{proof}
 Since for any $f \in \ell^\infty(X)$ and $t$ a partial translation of $X$, we have the identity $f V_t = V_t (f \circ t)$, where 
 \[
  (f \circ t)(x) = \begin{cases}
                  f(t(x)) & \text{if\ } x \in \mathrm{dom}(t) \\
                  0  & \text{if\ } x \not\in \mathrm{dom}(t) \\
                 \end{cases}
 \]
 for any $x \in X$, thus any operator $T \in C^*_\mathrm{u,alg}(X)$ may be written as a finite sum $\displaystyle \sum_{i = 1}^n V_{t_i} f_i$ for partial translations $t_1, \ldots, t_n$ of $X$ and functions $f_1 ,\ldots, f_n \in \ell^\infty(X)$. We define a norm $\|\cdot\|_\mathrm{PT}$ on $C^*_\mathrm{u,alg}(X)$ as
 \[
  \|\ T \|_\mathrm{PT} := \inf \left\{ \sum_{i = 1}^n \|f_i\| \colon T = \sum_{i = 1}^n V_{t_i} f_i,\ t_1, \ldots, t_n \in \mathrm{PT}(X),\ f_1 ,\ldots, f_n \in \ell^\infty(X), \ n \in \N \right\}
 \]
 for any $T \in C^*_\mathrm{u,alg}(X)$; recall from Definition~\ref{def:part-transl} that $\mathrm{PT}(X)$ denotes the set of partial translations of $X$.
 Note that $\|\cdot\|_\mathrm{PT}$ is in general not a $C^*$-norm, but will give an upper bound for any $C^*$-norm:
 since any nonzero partial isometry has norm 1 under any $C^*$-norm, we see that any $*$-representation of $\C_\mathrm{u}(X)$ is a bounded linear operator with regard to $\|\cdot\|_\mathrm{PT}$ (restricted to $\C_\mathrm{u}(X)$). On the other hand, since every function in $\ell^\infty(X)$ may be approximated by step functions, we see that $\C_\mathrm{u}(X)$ is dense in $C^*_\mathrm{u,alg}(X)$ under $\|\cdot\|_\mathrm{PT}$. Therefore any $*$-representation of $\C_\mathrm{u}(X)$ extends uniquely to $C^*_\mathrm{u,alg}(X)$ up to unitary equivalence. 
\end{proof}

\begin{remark}\label{rem:uniform-Roe-definition}
	In the context of Roe algebras it is also convenient to define the \emph{propagation} of an operator $T\cong[T_{xy}]_{x,y\in X}\in \mathcal{L}(\ell^2(X))$ by
	\begin{align*}
	p(T):=\sup\Big\{d(x,y):x,y\in X\quad\mathrm{and}\quad T_{xy}\neq 0\Big\}.
	\end{align*}
	It is clear, due to the controlledness of the partial translations (cf., Definition~\ref{def:part-transl}),
	that every partial isometry coming from a partial translation has finite propagation. The elements in 
	$\ell^\infty(X)$, in particular, all characteristic functions $P_F$, $F\subset X$, have zero propagation. 
	Consequently, all the elements in the algebraic uniform Roe-algebra $C^*_\mathrm{u,alg}(X)$ have finite propagation. The converse is also true, i.e.,
	$C^*_\mathrm{u,alg}(X)$ contains all the operators in $\cL(\ell^2(X))$ with finite propagation (cf., \cite[Chapter 4]{Roe03}).
\end{remark}

\begin{definition}\label{def:Roe-u-max}
The \emph{uniform Roe algebra} $C^*_\mathrm{u}(X)$ is defined as the closure of $\C_\mathrm{u}(X)$ (or equivalently, $C^*_\mathrm{u,alg}(X)$) in
$\cL(\ell^2(X))$. The \emph{maximal uniform Roe algebra} $C^*_\mathrm{u,max}(X)$ is defined as the universal enveloping $C^*$-algebra of
$\C_\mathrm{u}(X)$ (or equivalently, $C^*_\mathrm{u,alg}(X)$).
\end{definition}

\begin{example}\label{ex:CP}
It is well-known that every countable discrete group admits a proper (i.e., balls are compact) left-invariant metric, which is unique up to coarse equivalence (see
Lemma~2.1 in \cite{Tu01}). If the group is finitely generated and discrete, one can simply take the word metric associated to any finite set of
generators. In the case that the metric space is a discrete and countable group $\Gamma$ its associated uniform Roe algebra is canonically
isomorphic with the reduced crossed product
\[
 C_\mathrm{u}^*(\Gamma)=\ell^\infty(\Gamma)\rtimes_r\Gamma\;,
\]
where $\Gamma$ acts on $\ell^\infty(\Gamma)$ by left translations. Since every abelian $C^*$-algebra is F\o lner (in fact quasidiagonal)
it follows from Proposition~2.12 of \cite{Bedos95}, that $C_\mathrm{u}^*(\Gamma)$  is a F\o lner $C^*$-algebra if and only if the group
$\Gamma$ is amenable (see also Proposition~\ref{pro:crossedprods}).
\end{example}

\begin{example}\label{ex:semi-Roe}
We consider here Roe algebras as semi-pre-$C^*$-algebras (see Subsection~\ref{subsec:semi-prec}).
Let $X$ be a metric space with bounded geometry and $C^*_{\text{u,alg}}(X)$ be its algebraic uniform Roe algebra over $\C$. 
Then since the generators $f \in \ell^\infty(X)$ and $V_t$, for partial translations $t$ on $X$, 
all clearly belong to the subalgebra of bounded elements, 
we see that $C^*_{\text{u,alg}}(X)$ satisfies the Combes axiom, and is thus a 
semi-pre-$C^*$-algebra with respect to the algebraic $*$-positive cone $C^*_{\text{u,alg}}(X)_+$.
It follows directly from the definition that $C_{\text{univ}}^*(C^*_{\text{u,alg}}(X))=C^*_{u,\text{max}}(X)$, 
the maximal uniform Roe algebra of $X$.
If the algebraic uniform Roe algebra $C^*_{\text{u,alg}}(X)$ is equipped with the following archimedean closed $*$-positive cone
$\C_{u,\text{r}}(X)_+:=C^*_{\text{u,alg}}(X)\cap \cL(\ell^2(X))_+$, then the resultant semi-pre-$C^*$-algebra $\C_{u,\text{r}}(X)$ satisfies
$C^*_{\text{univ}}(\C_{u,\text{r}}(X))=C_\mathrm{u}^*(X)$, the uniform Roe algebra of $X$, which follows easily from Example~\ref{basis}. Moreover, 
Theorem~\ref{semi} and Example~\ref{basis} imply that $C^*_{u,\text{max}}(X)=C^*_{u}(X)$ if and only if
$\text{arch}(C^*_{\text{u,alg}}(X)_+)=\C_{u,\text{r}}(X)_+$. It is well-known that if $X$ has property A, 
then $C^*_{u,\text{max}}(X)=C^*_{u}(X)$. The converse implication is open in general (see \cite{SW13}).
\end{example}

Next we generalize the group case given in Example~\ref{ex:CP} to general metric spaces with bounded geometry
(see also \cite[Theorem~4.6]{Roe03}, \cite[Proposition~5.5]{RS12} and  \cite[Theorem~1]{Elek97}).
We give a multitude of equivalent conditions to characterize amenability. To formulate some of the characterizations, we need to recall a few concepts.

\begin{definition}
	We say that two idempotents $e$ and $f$ in an algebra $\cA$ are equivalent and write $e\sim f$, 
	if there are $x,y\in \cA$ such that $e=xy$ and $f=yx$. 
	An idempotent $e$ in an algebra $\cA$ is said to be {\it properly infinite} if there are mutually 
	orthogonal idempotents $e_1, e_2$ in $e\cA e$ such that $e_1 \sim e\sim e_2$.
	A (nonzero) unital algebra $\cA$ is said to be {\it properly infinite} if $\1$ is a properly infinite idempotent. 
\end{definition}

\begin{definition}[see also \cite{Abrams15,Leav,ALLW-1} for additional results and motivation]
	Let $n,m$ be integers such that $1\le m< n$. Then the Leavitt algebra (over $\C$)
	$L(m,n)$ is the algebra generated by elements $X_{ij}$ and $Y_{ji}$,
	for $i=1,\dots,m$ and $j=1,\dots,n$, such that $XY = \1_m$ and $YX=
	\1_n$, where $X$ denotes the $m\times n$ matrix $(X_{ij})$ and $Y$
	denotes the $n\times m$ matrix $(Y_{ji})$. The algebra $L_{\infty }$ is the unital algebra generated by 
	$X_1,Y_1,X_2, Y_2, \dots $ subject to the relations $Y_jX_i = \delta_{i,j} \1$.
\end{definition}

We remark that the algebras $L(m,n)$ are simple if and only if $m=1$ (\cite[Theorems 2 and 3]{LeavDuke}) and the algebra $L_{\infty }$ is simple (\cite[Theorem 4.3]{AGP}). 

\begin{theorem}\label{theorem:main1}
Let $(X,d)$ be an extended metric space with bounded geometry, $\cA$ be a unital pre-$C^*$-algebra that contains the translation algebra $\C_\mathrm{u}(X)$ 
as a dense $*$-subalgebra. Let $k,n$ be positive integers with $n\geq 2$.
Then the following conditions are equivalent: 
\begin{enumerate}
 \item \label{theorem:main1:coarse} $(X,d)$ is amenable.
 \item \label{theorem:main1:algebraic} $C^*_\mathrm{u,alg}(X)$ is algebraically amenable. 
 \item \label{theorem:main1:Foelner} $\cA$ is a F\o lner pre-$C^*$-algebra. 
 \item \label{theorem:main1:state} $\cA$ has a tracial state.
 \item \label{theorem:main1:stably-properly-infinite} $M_k(\cA)$ is not properly infinite. 
 \item \label{theorem:main1:Leavitt-A} $\cA$ does not contain the Leavitt algebra $L(1,n)$ as a unital $*$-subalgebra.
 \item[(6')] \label{theorem:main1:Leavitt-A-a} $\cA$ does not contain the Leavitt algebra $L(1,2)$ as a unital $*$-subalgebra.
 \item[(6'')] \label{theorem:main1:Leavitt-A-b} $\cA$ does not contain the Leavitt algebra $L_\infty$ as a unital $*$-subalgebra.
 \item \label{theorem:main1:K-theory-A} $[1]_0\neq [0]_0$ in the (algebraic) $K$-group $K_0(\cA)$.
\end{enumerate}
\end{theorem}

\begin{proof} 
(\ref{theorem:main1:coarse}) $\Rightarrow$ (\ref{theorem:main1:algebraic}): 
For any $\varepsilon>0$ and any finite set $\cF\subset C^*_\mathrm{u,alg}(X)$ let $R$ be an upper bound for the propagation of operators in $\cF$. 
Choose an $(R,\varepsilon)$-F\o lner set $F\subset X$ and consider the subspace (in fact, subalgebra) given by $W:=P_F C^*_\mathrm{u,alg}(X) P_F$.
Since $F$ is finite we have that $P_F C^*_\mathrm{u,alg}(X) P_F=P_F \C_\mathrm{u}(X) P_F$ and the proof that 
$W$ implements algebraic amenability follows from Theorem~6.3 in \cite{ALLW-1} taking as field $\K=\C$.

(\ref{theorem:main1:algebraic}) $\Rightarrow$ (\ref{theorem:main1:Foelner}): This implication follows directly from Example~\ref{basis} and Theorem~\ref{theorem:alg-amen-to-Folner}(i).

(\ref{theorem:main1:Foelner}) $\Rightarrow$ (\ref{theorem:main1:state}): This follows from Theorem~\ref{theorem:characFolner2}:
there exists a faithful (unital) representation $\pi$ of $\cA$ such that $\pi(\cA)$ has an amenable trace.

(\ref{theorem:main1:state}) $\Rightarrow$ (\ref{theorem:main1:stably-properly-infinite}): Recall that a projection $P$ in a pre-$C^*$-algebra 
$\cA$ is properly infinite if there exist partial isometries $X$,$Y$ in $\cA$ such that $X^*X=Y^*Y=P$ and satisfying $X^*Y=Y^*X=0$ as well as
$XX^*+YY^*\leq P$. A unital $*$-algebra is properly infinite if the unit is properly infinite.
Assume that $M_k(\cA)$ is properly infinite for some $n\in \N$. Then there are isometries 
$X$ and $Y$ in $M_k(\cA)$ such that $X^*X=Y^*Y=\1_k$, $X^*Y=Y^*X=0$ and $XX^*+YY^*\leq \1_k$. Let $\tau:\cA\rightarrow \C$ 
be a tracial state and define a positive trace $\tau_k\colon M_k(\cA)\rightarrow \C$ by
\begin{align*}
\tau_k\big([X_{ij}]\big)=\sum_{i=1}^n\tau(X_{ii}).
\end{align*}
It is clear that $\tau_k(\1_k)=n$. Hence $2k=\tau_k(XX^*+YY^*)\leq \tau_k(\1_k)=k$, which is a contradiction.

(\ref{theorem:main1:stably-properly-infinite}) $\Rightarrow$ (\ref{theorem:main1:Leavitt-A}): 
Assume that $L(1,n)$ is a unital $*$-subalgebra of $\cA$. Then it follows from the definition of the 
Leavitt algebra that the unit in $\cA$ is a properly infinite projection, hence also $M_k(\cA)$ is properly infinite.

(\ref{theorem:main1:Leavitt-A}) $\Rightarrow$ (\ref{theorem:main1:Leavitt-A}'): This follows from the standard unital embedding 
$L(1,n) \hookrightarrow L(1,2)$ as $*$-algebras: 
the generators $X_{11}, X_{12}, X_{13}, \ldots, X_{1n} , Y_{11}, Y_{21}, Y_{31}, \ldots, Y_{n1}$ of $L(1,n)$ are mapped to, respectively, 
\begin{align*}
 \widetilde{X}_{11}^{n-1}\,, && \widetilde{X}_{11}^{n-2} \widetilde{X}_{12}\,, && \widetilde{X}_{11}^{n-3} \widetilde{X}_{12}\,, && \ldots\,, && \widetilde{X}_{11} \widetilde{X}_{12}\,, && \widetilde{X}_{12},\\ 
 \widetilde{Y}_{11}^{n-1}\,, && \widetilde{Y}_{21} \widetilde{Y}_{11}^{n-2}\,, && \widetilde{Y}_{21} \widetilde{Y}_{11}^{n-3}\,, && \ldots\,, && \widetilde{Y}_{21} \widetilde{Y}_{11}\,, && \widetilde{Y}_{21} \,,
\end{align*}
 where $\widetilde{X}_{11}, \widetilde{X}_{12}, \widetilde{Y}_{11}, \widetilde{Y}_2$ are the canonical generators of $L(1,2)$
(see also \cite{BS16}).

(\ref{theorem:main1:Leavitt-A}') $\Rightarrow$ (\ref{theorem:main1:coarse}): If $X$ is not amenable, then by Theorem~\ref{theorem:amenable-metric}~(2) 
there exist a partition $X=X_1\sqcup X_2$ and two partial
translations $t_i\colon X\rightarrow X_i$. Let $V_{t_i}$ be the isometry corresponding to $t_i$, $i=1,2$, which satisfy
$
 V_{t_1}V_{t_1}^*+V_{t_2}V_{t_2}^*=\1
$.
It follows from the universal property and simplicity of the Leavitt algebras that $L(1,2)$ is
$*$-isomorphic to the unital $*$-subalgebra generated by $V_{t_1}$ and $V_{t_2}$.

(\ref{theorem:main1:Leavitt-A}'') $\Rightarrow$ (\ref{theorem:main1:Leavitt-A-a}') and (\ref{theorem:main1:stably-properly-infinite}) $\Rightarrow$ (\ref{theorem:main1:Leavitt-A}''): these implications follow from the fact that
a unital algebra is properly infinite if and only if there is a unital embedding $L_{\infty }\hookrightarrow \cA$
(see Proposition~5.2 in \cite{ALLW-1}). Finally, it remains to make contact with the last condition.

(\ref{theorem:main1:state}) $\Rightarrow$ (\ref{theorem:main1:K-theory-A}): If $\tau\colon \cA\rightarrow \C$ is a tracial state, then it follows from the universal property of $K_0$ that $\tau$ induces a (unique) group homomorphism 
\[ 
\tau_*\colon K_0(\cA)\rightarrow \R \quad\mathrm{with}\quad \tau_*([P]_0)=\tau(P)\quad\mathrm{for~any~idempotent~}\;P\in\cA′\;.
\]
As $\tau(\1) = 1$, we have $[\1]_0\neq [0]_0$ in $K_0(\cA)$.

(\ref{theorem:main1:K-theory-A}) $\Rightarrow$ (\ref{theorem:main1:Leavitt-A}'): 
Suppose that the Leavitt algebra $L(1,2)$ is a unital subalgebra of $\cA$. Let $X_1$,$X_2$, $Y_1$,$Y_2$ be the canonical generators of $L(1,2)$. Then the idempotents $P = X_1 Y_1$ and $Q = X_2 Y_2$ sum up to $\1$, while each is also algebraically equivalent to $\1$ in $\cA$. Thus,
\begin{align*}
[\1]_0 = [P+Q]_0= [P]_0 + [Q]_0 = [\1]_0 + [\1]_0 
\end{align*}
and we conclude that $[\1]_0=[0]_0$ in $K_0(\cA)$.
\end{proof}

An immediate consequence of the preceding equivalences is the following result.
\begin{corollary}
Let $\cA$ be the uniform Roe algebra $C_\mathrm{u}^*(X)$ or the maximal uniform Roe-algebra $C_\mathrm{u,max}^*(X)$.
Then $\cA$ is a F\o lner $C^*$-algebra if and only if the Cuntz algebra $\cO_n$, $n\geq 2$, does not unitally embed into $\cA$
if and only if $(X,d)$ is an amenable metric space.
\end{corollary}

\begin{remark}
  Note that by virtue of Lemma~\ref{lem:same-representations}, the result of the preceding theorem remains true if we replace $C^*_{\text{u,alg}}(X)$ in 
  (\ref{theorem:main1:algebraic}) by a 
  subalgebra (not necessary a $*$-subalgebra) $\cC$ of $C^*_{\text{u,alg}}(X)$ that contains the 
  translation algebra $\C_\mathrm{u}(X)$.
\end{remark}

We also have an analogous result linking various notions of proper amenability (see also Definition~\ref{def:properly-Foelner}):
\begin{theorem}\label{teo:Roe-proper-amenability}
 Let $(X,d)$ be an extended metric space.
 Then the following conditions are equivalent: 
 \begin{enumerate}
 \item \label{teo:Roe-proper-amenability:coarse} $(X,d)$ is properly amenable. 
 \item \label{teo:Roe-proper-amenability:algebraic} $C^*_\mathrm{u,alg}(X)$ is properly algebraically amenable. 
 \item \label{teo:Roe-proper-amenability:Foelner} $C^*_\mathrm{u}(X)$ is a properly F\o lner $C^*$-algebra.
\end{enumerate}
\end{theorem}

\begin{proof}
 (\ref{teo:Roe-proper-amenability:coarse}) $\Rightarrow$ (\ref{teo:Roe-proper-amenability:algebraic}): 
 This follows as in Theorem~\ref{theorem:main1} (see also Theorem~6.3 in \cite{ALLW-1}).
 
 (\ref{teo:Roe-proper-amenability:algebraic}) $\Rightarrow$ (\ref{teo:Roe-proper-amenability:Foelner}): 
 This follows directly from Theorem~\ref{theorem:alg-amen-to-Folner}(ii).
 
 (\ref{teo:Roe-proper-amenability:Foelner}) $\Rightarrow$ (\ref{teo:Roe-proper-amenability:coarse}): 
 Since $\cA : = C^*_\mathrm{u}(X) $ is properly F\o lner, and in particular F\o lner, we know from Theorem~\ref{theorem:main1} that $X$ is amenable. 
 Now suppose $X$ were not properly amenable hence not finite. 
By Corollary~\ref{cor:characamenable-nprpamen}, there would be a decomposition $X = Y_1 \sqcup Y_2$ 
with $Y_1$ being finite and non-empty, $Y_2$ being non-amenable, and $d(y_1, y_2) = \infty$ for $y_1 \in Y_1$, $y_2 \in Y_2$. 
We obtain therefore two central projections $P_{Y_1}$ and $P_{Y_2}$ in 
$\C_\mathrm{u}(X)$ which add up to the unit. Note that they are also central in $\cA$. 
Thus setting $\cA_i = P_{Y_i} \cA P_{Y_i}$ for $i = 1,2$, we arrive at a direct sum decomposition 
$\cA = \cA_1 \oplus \cA_2$.  
Since $\cA_1$ is finite-dimensional, Proposition~\ref{prop:properly-Foelner-direct-sum-B} tells us that 
 $\cA_2 = C^*_\mathrm{u}(Y_2)$ is properly F\o lner. But by Theorem~\ref{theorem:main1}, 
that implies the amenability of the space $Y_2$, which is contradictory to our assumption about $Y_2$. 
\end{proof}

The following result shows that the F\o lner sequence for an operator in the uniform Roe-algebra can be chosen 
within the translation algebra.

\begin{corollary}\label{cor:inner-sequence}
Let $(X,d)$ be an extended amenable metric space with bounded geometry. 
Then any operator $T\in C_\mathrm{u}^*(X)$ has a F\o lner sequence which can be chosen within the translation algebra
$\C_\mathrm{u}(X)$.
\end{corollary}
\begin{proof}
Let $T\in\C_\mathrm{u}(X)$ be of propagation $R$, $F\subset X$ and denote by $P_F$, the characteristic function over $F$.
Note that since $p(P_F)=0$, we have that $P_F\in\C_\mathrm{u}(X)$.
Recalling the notions of outer and inner boundary given in 
Subsection~\ref{subsec:amenability-metric}, one has for the commutator of $T$
with $P_F$
\[
[T, P_F] = P_{\partial^+_R F} T P_{\partial^-_R F} - P_{\partial^-_R F} T P_{\partial^+_R F}\;.
\]
Therefore, since $(X,d)$ is amenable, we can choose $F\in\mathrm{\mbox{F\o l}}\left(R,\frac{\varepsilon^2}{4\|T\|^2}\right)$ 
and
\[
 \frac{\|[T, P_F]\|_2}{\|P_F\|_2} =\frac{\| P_{\partial^+_R F} T P_{\partial^-_R F} - P_{\partial^-_R F} T P_{\partial^+_R F}\|_2}{\|P_F\|_2 }
                                                      \leq 2 \|T\|  \left( \frac{|\partial_R F|}{|F|}\right)^{\frac12} < \varepsilon\;.
\]
Finally, since $\C_\mathrm{u}(X)$ is dense in $C_\mathrm{u}^*(X)$ it follows that any operator in $C_\mathrm{u}^*(X)$ has a F\o lner sequence
belonging to the translation algebra $\C_\mathrm{u}(X)$.
\end{proof}

\begin{remark}\label{rem:property-A}
Recall that $(X,d)$ has property~A if and only if $C_\mathrm{u}^*(X)$ is a nuclear $C^*$-algebra (cf., \cite[Theorem 5.5.7]{bBrown08}). 
Nevertheless, from Remark~\ref{rem:A} it follows that $(X,d)$ having property~A is, in general, independent 
of the uniform Roe $C^*$-algebra $C_\mathrm{u}^*(X)$ being F\o lner $C^*$-algebra. For example,
$C_\mathrm{u}^*(\mathbb{F}_2)$ is not a F\o lner $C^*$-algebra and $X=\mathbb{F}_2$
has property~A. Moreover, the box space 
$X:=\text{Box}_{\Gamma_n}(\mathbb{F}_2)$  is an amenable metric space with bounded geometry
which does {\em not} have property A. By Theorem~\ref{theorem:main1} the corresponding uniform Roe algebra 
$C_\mathrm{u}^*(X)$ is a F\o lner $C^*$-algebra.
\end{remark}

\subsection{The trace space of uniform Roe algebras}\label{subsec:trace-space}

The analysis of the trace space of a $C^*$-algebra is a fundamental problem one faces in the study of $C^*$-algebras. 
Recall that, in the context of amenability for discrete groups, $\Gamma$ is amenable 
if and only if every trace of the reduced group $C^*$-algebra is amenable. Other important classes of $C^*$-algebras 
like nuclear $C^*$-algebras or $C^*$-algebras with the weak expectation property (WEP) also satisfy that
any trace is amenable (see, e.g., Chapters 3 and 4 in \cite{Brown06} for details).

From Theorems~\ref{theorem:characFolner2} and~\ref{theorem:main1}, we see that for a uniform Roe algebra, 
the existence of a trace is equivalent to that of an amenable trace. 
One may ask a more refined question: is every trace of a uniform Roe algebra amenable? 
We are going to answer this question affirmatively in this section. 

In the following we will assume for simplicity that $\cA$ is a $C^*$-algebra. Note that if $\C_\mathrm{u}(X)\subset \cA$,
then $\ell^\infty(X)\subset\cA$, since any bounded function on $X$ may be approximated in norm by step functions
lying in $\C_\mathrm{u}(X)$.

\begin{lemma}\label{lem:extension}
Let $\cA$ be a $C^*$-algebra that contains the translation algebra $\C_\mathrm{u}(X)$ as a dense subalgebra. Then:
\begin{enumerate}
\item The standard representation $\pi$ of $\C_\mathrm{u}(X)$ on $\cL(\ell^2(X))$ extends to $\cA$.
\item There is a conditional expectation $E\colon\cA\to\ell^\infty(X)$.
\end{enumerate}
\end{lemma}
\begin{proof}
(1) The statement is equivalent to that the identity map on $\C_\mathrm{u}(X)$ extends to a $*$-homomorphism from $\cA$ to $C^*_{u}(X)$. For this it suffices to show that the $C^*$-norm $\|\cdot\|_{\cA}$ restricted to $\C_\mathrm{u}(X)$ is greater than or equal to the $C^*$-norm $\|\cdot\|_\mathrm{r}$ given by the standard representation of $\C_\mathrm{u}(X)$. To this end we are going to show that for any $T \in \C_\mathrm{u}(X)$ with $\|T\|_\mathrm{r} = 1$ and $\varepsilon > 0$, we have $\|T\|_{\cA} \geq 1 - 2\varepsilon$. 

To begin with, we pick a unit vector $\xi \in \ell^2(X)$ such that $\|T \xi\|_2 \geq 1 - \varepsilon$. Now as $1 = \| \xi \|_2^2 = \sum_{x\in X} |\xi(x)|^2$, we may find a finite subset $F \subset X$ such that $\| P_F \xi - \xi \|_2 \leq \varepsilon$; thus 
\[
 \|T P_F \xi \|_2 \geq \|T \xi \|_2 - \|T (\xi - P_F \xi)\|_2 \geq (1- \varepsilon ) - \varepsilon = (1 - 2\varepsilon) \|\xi\|_2 \; .
\]
Hence $\| T P_F T^* \|_\mathrm{r} = \| T P_F \|_\mathrm{r}^2  \geq (1 - 2\varepsilon)^2$. 
Since $F$ is finite and by the properties of $C^*$-norms we have
\[
\| T P_F T^* \|_{\cA} =  \| P_F T^*T P_F \|_{\cA} =
\| P_F T^*T P_F \|_\mathrm{r} = \| T P_F T^* \|_\mathrm{r}  \geq (1 - 2\varepsilon )^2\;.
\]
As $T T^* \geq T P_F T^*$ in $\cA$, we have $\| T \|_{\cA} = \sqrt{ \| T T^*  \|_{\cA} } \geq \sqrt{ \| T P_F T^* \|_{\cA} } \geq 1 - 2\varepsilon$.

(2) The composition of the standard representation $\pi$ on $\cL(\ell^2(X))$ with the conditional expectation 
$E_0\colon \mathcal{L}(\ell^2(X)) \to \ell^\infty (X)$ gives a conditional expectation $E$ as required.
\end{proof}

\begin{lemma}\label{lem:factor-tracial-state}
Let $\cA$ be a $C^*$-algebra that contains the translation algebra $\C_\mathrm{u}(X)$ as a dense subalgebra and consider a tracial state
$\tau$ of $\cA$. Then $\tau = \tau|_{\ell^\infty (X)} \circ E$.
\end{lemma}

\begin{proof}
 We first prove $\tau (V_t) = \left(\tau|_{\ell^\infty (X)} \circ E\right) (V_t) = 0$ for any partial translation $t$ with the property that 
 $\mathrm{dom}(t) \cap \mathrm{ran}(t) = \varnothing$. Indeed, on the one hand, 
\[
 \tau (V_t) = \tau (P_{\mathrm{ran}(t)} V_t P_{\mathrm{dom}(t)}) = \tau (V_t P_{\mathrm{dom}(t)} P_{\mathrm{ran}(t)} ) = 0\;,
\] 
while, on the other hand, since $\langle \delta_x, V_t \delta_x \rangle = 0$ for any $x \in X$, we have $E(V_t) = 0$. 
 
Next we extend the result to any partial translation $t$ without a fixed point, i.e., such that for any 
$x \in \mathrm{dom}(t)$, $t(x) \not= x$. To this end, consider the oriented graph whose vertices are 
labeled by points in $\mathrm{dom}(t)$, so that for any $x, y \in \mathrm{dom}(t)$, there is an edge from $x$ to $y$ 
if and only if $y = t(x)$. Since $t$ is a partial translation, it follows that every vertex has at most one incoming edge and at most one outgoing edge, and since there is no fixed point, there is no self-loop in the graph. The standard greedy coloring algorithm in graph theory enables us to color the vertices of this graph with 3 colors, i.e., decompose $\mathrm{dom}(t)$ into a disjoint union $Y_1 \sqcup Y_2 \sqcup Y_3$, so that no vertices of the same color are adjacent. Let $t_i = t|_{Y_i}$ for $i = 1,2,3$. Then by our construction of the graph and the coloring scheme, we have for each $i \in \{1,2,3\}$, for any $x,y \in \mathrm{dom}(t_i)$, $t(x) \not = y$; thus $\mathrm{dom}(t_i) \cap \mathrm{ran}(t_i) = \varnothing$. Hence we may apply the previous case to the sum $V_t = V_{t_1} + V_{t_2} + V_{t_3}$ and prove $\tau (V_t) = \left(\tau|_{\ell^\infty (X)} \circ E \right)(V_t)=0$ by additivity.
 
Now for an arbitrary partial translation $t$, we let $X^t = \{x \in \mathrm{dom}(t) \ | \ t(x) = x\}$ be the set of its fixed points. Thus we can decompose $V_t$ as $P_{X^t} + V_{t'}$, where $t'$ is a partial translation without fixed points. Since $E(P_{X^t}) = P_{X^t}$, we have $\tau (P_{X^t}) = \tau|_{\ell^\infty (X)} \circ E (P_{X^t})$. The equality $\tau (V_t) = \tau|_{\ell^\infty (X)} \circ E (V_t)$ follows by additivity. Therefore we still have the claimed equality.
 
Since every element in $\C_\mathrm{u}(X)$ is a linear combination of partial isometries $V_t$ associated to partial translations $t$, 
the desired identity holds on $\C_\mathrm{u}(X)$. Finally, we extend the identity to $\cA$ by continuity. 
\end{proof}

For the next result recall that by one of the Riesz representation theorems, any mean $\mu$ on $(X,d)$ induces uniquely a state $\phi_\mu\colon \ell^\infty(X)\to \C$ such that $\mu(Y) = \phi_\mu(P_Y)$ for any subset $Y \subset X$. Moreover, it is immediate that $\mu$ is invariant under partial translations 
(cf., Definition~\ref{def:part-transl}) if and only if $\phi_\mu$ is invariant in the sense that $\phi_\mu(f) = \phi_\mu(f \circ t)$ for any partial translation $t$ and any $f \in \ell^\infty(X)$ supported in $\operatorname{ran}(t)$. It is clear that $\mu$ is also uniquely determined by $\phi_\mu$.

\begin{theorem}
 Let $X$ be an extended metric space with bounded geometry and $\cA$ be a $C^*$-algebra that contains the translation algebra 
 $\C_\mathrm{u}(X)$ as a dense $*$-subalgebra. Let $\tau$ be a linear functional on $\cA$ and consider the conditional expectation $E\colon\cA\to\ell^\infty(X)$. 
 Then the following conditions are equivalent:
 \begin{enumerate}
  \item $\tau$ is an amenable tracial state.
  \item $\tau$ is a tracial state.
  \item $\tau = \phi_\mu \circ E$ for a mean $\mu$ on $(X,d)$ which is invariant under partial translations.
 \end{enumerate}
\end{theorem}

\begin{proof}
 (1) $\Rightarrow$ (2) is obvious. To see (2) $\Rightarrow$ (3), we apply the Lemma~\ref{lem:factor-tracial-state} 
 and define $\phi := \tau|_{\ell^\infty (X)}$. That $\phi$ is a normalized positive functional follows from the fact that 
 $\tau$ is a state. Thus the formula $\mu(C) = \phi(P_C)$ for $C \subset X$ defines a mean on $X$. Invariance is a consequence of the trace property: let $(A,B,t)$ be a partial translation on $X$, then
\[ 
\mu(A)=\phi_\mu(P_A)=\phi(P_A)=\tau(V_t^* V_t)= \tau(V_t V_t^*) =\phi_\mu(P_B)=\mu(B)\;.
\]
 
To prove (3) $\Rightarrow$ (1), let us consider the standard representation $\pi$ of $\cA$ on $\cL(\ell^2(X))$ and the state $\psi = \phi \circ E_0$,
where $E_0$ is the conditional expectation from $\cL(\ell^2(X))$ onto $\ell^\infty(X)$. By Lemma~\ref{lem:extension}~(ii) we have $\tau = \psi \circ
\pi$. It remains to show that $\psi$ is a hypertrace for $\pi(\cA)$, i.e., $\psi (T \pi(a)) = \psi (\pi(a) T)$ for any $a \in \cA$ and $T \in
\cL(\ell^2(X))$ (cf., Definition~\ref{def:amenable}). It is enough to assume that $\pi(a)$ is an arbitrary partial isometry $V_t$, where $t$ is an arbitrary partial
translation on $X$. Using Equation~(\ref{partial-trans}), we obtain
 \begin{align*}
  E_0(V_t T) (x) & \ = \begin{cases}
                  T_{t(x)\,,\,x}  &, \ \mathrm{if}\ x \in \mathrm{dom}(t)\\
                  0 &, \ \mathrm{if}\ x \notin \mathrm{dom}(t)
                 \end{cases}\\
  E_0(T V_t) (x) & \ = \begin{cases}
                  T_{x\,,\, t^{-1}(x)} &, \ \mathrm{if}\ x \in \mathrm{ran}(t)\\
                  0 &, \ \mathrm{if}\ x \notin \mathrm{ran}(t)
                 \end{cases}        
 \end{align*}
 for any $x \in X$. Thus $\mathrm{supp}(E_0(V_t T)) = \mathrm{dom}(t)$, $\mathrm{supp}(E_0(T V_t)) = \mathrm{ran}(t)$ and $E_0(V_t T) = E_0(T V_t) \circ V_t$. Since $\phi$ is invariant, we have 
$\phi (E_0(V_t T)) = \phi (E_0(T V_t))$, which shows that $\psi$ is a hypertrace.
\end{proof}

We conclude by highlighting a direct consequence of the previous results and a natural question that arises out of it.
\begin{corollary}\label{amenable trace}
Let $(X,d)$ be an extended metric space with bounded geometry and denote by $\cA$ the uniform Roe algebra $C_\mathrm{u}^*(X)$ or the maximal uniform Roe-algebra $C_\mathrm{u,max}^*(X)$. Then $\cA$ has a tracial state if and only if $(X,d)$ is amenable. In this case every tracial state of $\cA$ is amenable.
\end{corollary}

Since quasidiagonality is a strengthening of amenability for a tracial state, we ask:

\begin{problem}\label{final question}
 Let $(X,d)$ be an extended metric space with bounded geometry and denote by $\cA$ the uniform Roe algebra $C_\mathrm{u}^*(X)$ or the maximal uniform
 Roe-algebra $C_\mathrm{u,max}^*(X)$. Then is every tracial state of $\cA$ quasidiagonal?
\end{problem}

We remark that in a breakthrough in the structure theory of $C^*$-algebras with the UCT (the universal coefficient theorem for $KK$-theory), 
Tikuisis, White and Winter (\cite{TWW15}) showed that all faithful traces on a separable nuclear $C^*$-algebra with UCT are quasidiagonal, while the question
of quasidiagonality of amenable traces on nonseparable algebras was subsequently considered by Gabe (\cite{Gabe17}). 

\begin{remark}
We would like to thank the anonymous referee for suggesting to us a nice argument showing that Problem~\ref{final question} has an affirmative answer if we assume that $(X,d)$ has property A. In this case, the maximal and the reduced uniform Roe algebras coincide. The argument from the referee uses the fact that $C_\mathrm{u}^*(X)$ is the $C^*$-algebra of the coarse groupoid $G(X)$ as defined in \cite{STY02}. The groupoid $G(X)$ is amenable in the presence of property A and may be written as a crossed product groupoid $\beta X \rtimes H$ for a metrizable $\sigma$-compact \'{e}tale groupoid $H$ (see \cite[Lemma~3.3]{STY02}). By \cite[Theorem 4.4]{CW05} and \cite[Theorem 6.3]{CW04}, for a tracial state $\tau$ of $C_\mathrm{u}^*(X)$, the quotient of $C_\mathrm{u}^*(X)$ by the trace kernel $I_\tau$ is canonically isomorphic to the $C^*$-algebra of an amenable groupoid $Y \rtimes H$, namely, the reduction of $\beta X \rtimes H$ to a closed invariant subset $Y$ of the unit space $\beta X$. Because the amenability of an $\sigma$-compact \'{e}tale groupoid is witnessed by a countable sequence of functions on it (see \cite[Definition~5.2]{STY02}), we may write $Y$ as $\displaystyle \varprojlim_{i} Y_i$, with each $Y_i$ being a metrizable $H$-invariant quotient of $Y$ such that $Y_i \rtimes H$ remains amenable. It follows that $\displaystyle C_\mathrm{u}^*(X) / I_\tau = \varinjlim_{i} C^*(Y_i \rtimes H)$, with each $C^*(Y_i \rtimes H)$ being separable and nuclear and satisfying the UCT (by \cite[Proposition 10.7]{Tu99}). Now $\tau$ factors through a faithful tracial state $\tau_0$ on $C_\mathrm{u}^*(X) / I_\tau$. Hence each $\displaystyle \tau_0|_{C^*(Y_i \rtimes H)}$ is quasidiagonal by \cite{TWW15}. It follows by the local nature of quasidiagonality that $\tau_0$ is also quasidiagonal, and thus so is $\tau$. 
\end{remark}

\vspace{1cm}

\paragraph{\bf Acknowledgements:}
The second-named author is partially supported by the DFG (SFB 878) and he wishes to thank James Gabe and Mikael R\o rdam 
for helpful discussions on the subject of properly infinite $C^*$-algebras.
The third-named author thanks Wilhelm Winter for his kind invitation to the Mathematics 
Department of the University of M\"unster in April 2014 and March-June 2016.
Finantial support was provided by the DFG through SFB 878, as well as, by a DAAD-grant during these visits. Part of the research was conducted during visits and workshops at Universitat Aut\`onoma de Barcelona, University of Copenhagen, University of M\"{u}nster and Institut Mittag-Leffler. The authors owe many thanks and great appreciation to these institutes and hosts for their hospitality. 


\providecommand{\bysame}{\leavevmode\hbox to3em{\hrulefill}\thinspace}

\end{document}